\documentclass[pdflatex,sn-mathphys]{sn-jnl}

\usepackage{graphicx}%
\usepackage{multirow}%
\usepackage{amsmath,amssymb,amsfonts}%
\usepackage{amsthm}%
\usepackage{mathrsfs}%
\usepackage[title]{appendix}%
\usepackage{xcolor}%
\usepackage{textcomp}%
\usepackage{manyfoot}%
\usepackage{booktabs}%
\usepackage{algorithm}%
\usepackage{algorithmicx}%
\usepackage{algpseudocode}%
\usepackage{listings}%

\theoremstyle{thmstyleone}%
\newtheorem{theorem}{Theorem}
\newtheorem{proposition}[theorem]{Proposition}%

\newtheorem{lemma}{Lemma}

\theoremstyle{thmstyletwo}%
\newtheorem{remark}{Remark}%

\theoremstyle{thmstylethree}%
\newtheorem{definition}{Definition}%

\usepackage{placeins}

\usepackage{array}

\allowdisplaybreaks

\begin{document}

\title[On the moments of  Laguerre-Hahn linear forms of class zero]{On the moments of  Laguerre-Hahn linear forms of class zero}


\author[1]{\fnm{Mohamed} \sur{Khalfallah}}\email{mohamed.khalfallah@fsm.rnu.tn}

\author[2]{\fnm{Pascal} \sur{Maroni}} 
\equalcont{17th January 1933 -- 16th January  2024}

\author*[3]{\fnm{Zélia} \sur{da Rocha}}\email{mrdioh@fc.up.pt}

\affil[1]{\orgdiv{Department of Mathematics}, 
\orgname{Faculty of Sciences of Monastir, University of Monastir}, 
\orgaddress{
\city{Monastir}, 
\postcode{5019}, 
\country{Tunisia}}}

\affil[2]{\orgdiv{Laboratoire Jacques-Louis Lions}, 
\orgname{Sorbonne Université, CNRS}, 
\orgaddress{
\street{Boite courrier 187; 4, place Jussieu}, 
\city{Paris}, 
\postcode{75252 Paris cedex 05}, 
\country{France}}}

\affil[3]{\orgdiv{Departamento de Matemática, Centro de Matemática da Universidade do Porto (CMUP)}, 
\orgname{Faculdade de Ciências da Universidade do Porto}, 
\orgaddress{
\street{Rua do Campo Alegre n. 687}, 
\city{Porto}, 
\postcode{4169-007}, 
\country{Portugal}}}


\abstract{The purpose of this paper is to provide a constructive method for generating the moments of Laguerre-Hahn linear forms of class zero. The approach relies on a general algebraic system on the moments, equivalent to the Laguerre-Hahn functional equation for forms of arbitrary class, from which the restriction to class zero yields explicit recurrences allowing the entire moment sequence to be built recursively from the first moment. These recurrences are applied to each of the ten canonical class-zero families, and the corresponding first moments are derived explicitly via a {\it Mathematica$^{\circledR}$} implementation.}




\maketitle


\section{Introduction}
Orthogonal polynomials occupy a central place in methods of numerical analysis and approximation as well as in applications of mathematical physics, ranging from quadrature methods, rational approximation, spectral theory, random matrix theory, and quantum mechanics. Their study is intimately connected to the linear functionals, or forms, on the space of polynomials: to each regular form $u$, there exists a unique monic orthogonal polynomial sequence (MOPS) $\{P_n\}_{n \geq 0}$, characterized by a linear second-order recurrence relation
\begin{equation*}
  P_{n+2}(x) = (x - \beta_{n+1})\,P_{n+1}(x) - \gamma_{n+1}\,P_n(x), \quad n \geq 0,
\end{equation*}
with initial conditions $P_0(x):=1$ and $P_1(x):=x-\beta_0$, where $\{\beta_n\}_{n \geq 0}$ and $\{\gamma_{n+1}\}_{n \geq 0}$ are sequences of complex numbers with $\gamma_{n+1} \neq 0$ for all $n \geq 0$.\\ 

Among the various sets of regular forms, the Laguerre-Hahn forms constitute a broad and natural generalization of the classical forms --- Hermite, Laguerre, Bessel, and Jacobi --- as well as of the semiclassical forms; see~\cite{Maroni-1991, Maroni-1994}. \\

Laguerre--Hahn orthogonal polynomials have been studied extensively in the literature; see \cite{Alaya-these-1996, Alaya-Maroni,Dini-these-1988,Dzoumba-these-1985,Magnus-1983, Maroni-1983, Maroni-1991} and the references therein. The main contributions are due to A.~P.~Magnus, P.~Maroni, and A.~Ronveaux, who analyzed the differential properties of these polynomials, characterized the Laguerre--Hahn character, and explored computational aspects of differential equations. A recent paper \cite{Rebocho-2026} provides a historical synthesis of their work and discusses further connections within the framework of integrable systems. Despite more recent advances, the foundations of the theory remain largely due to these authors.\\

A regular form $u$ is said to be a Laguerre-Hahn form if its formal Stieltjes function
\begin{equation*}
  S(u)(z) := -\sum_{n \geq 0} \frac{(u)_n}{z^{n+1}},
\end{equation*}
where $(u)_n=\langle u, x^n \rangle$ are its moments,
satisfies a Riccati differential equation
\begin{equation*}
  \Phi(z)\,S'(u)(z) = B(z)\,S^2(u)(z) + C(z)\,S(u)(z) + D(z),
\end{equation*}
where $\Phi$, $B$, $C$, and $D$ are polynomials. Equivalently, $u$ satisfies the
distributional functional equation
\begin{equation}\label{eq:functional}
  (\Phi\,u)' + \psi\,u + B\,(x^{-1}u^2) = 0,
\end{equation}
where $C = -\Phi' - \psi$ and $D$ is determined by $\Phi$, $\psi$, $B$, and $u$. The class of a Laguerre-Hahn form $u$ is the non-negative integer
\begin{equation*}
  s := \min\,\max\!\bigl(\deg\psi - 1,\;\max\{\deg\Phi,\,\deg B\} - 2\bigr),
\end{equation*}
where the minimum is taken over all admissible triples $(\Phi,\psi,B)$ for which
\eqref{eq:functional} holds. The class measures the minimal complexity of the defining
equation and induces a natural hierarchy within the Laguerre-Hahn family. \\

Forms of class zero --- which correspond to $\deg\Phi \leq 2$, $\deg\psi \leq 1$,
$\deg B \leq 2$ with $|c_2|+|a_1|+|b_2|>0$ --- represent the simplest non-semiclassical
set, and encompass ten canonical families organized as analogues of the Hermite,
Laguerre, Bessel, and Jacobi classical cases.
With the aim of describing this class of Laguerre--Hahn forms, a complete description was recently established in~\cite{Article-2-Soummi}, updating the one found in \cite{Bouakkaz-these, Bouakkaz-Maroni-1991}. More precisely, up to an affine change of variable, there are ten canonical cases, including two new ones analogous to the classical Bessel families that did not appear previously in the literature \cite{Bouakkaz-these, Bouakkaz-Maroni-1991}. By applying the constructive method and algorithm developed in~\cite{Article-1-Soummi}, structure relations and a fourth-order linear differential equation satisfied by the corresponding orthogonal polynomial sequences were derived for some of these families, see ~\cite{Article-1-Soummi, Article-2-Soummi}.

In the same spirit, the present paper addresses the problem of determining the moment sequences of these families.\\

A fundamental characteristic element of any linear form $u$ is its moment sequence $(u)_n$, $n \geq 0$ as it determines $u$ uniquely --- via the formal Stieltjes function expansion above --- and encodes all relevant information about the form:  the regularity through its recurrence coefficients $\beta_n$ and $\gamma_{n+1}$, and the analytic type through the structure relations satisfied by the corresponding MOPS. \\

The moments of a linear functional play a central role in the theory of orthogonal polynomials and encode its entire algebraic structure. Introduced by Stieltjes and systematized by Shohat and Tamarkin~\cite{Shohat1943}, then further developed by Chihara~\cite{Chihara-1978}, they determine the coefficients of the linear
second-order recurrence relation through ratios of Hankel determinants, and the quasi-definiteness of the functional is characterized precisely by the non-vanishing of all such determinants~\cite{Chihara-1978,Gautschi2004}. 
In numerical analysis, Golub and Welsch~\cite{GolubWelsch1969} established that Gauss quadrature rules are obtained directly from the moments via the eigenstructure of the associated Jacobi matrix, while the modified Chebyshev method \cite{SackDonovan1971} (see also \cite{Rocha-1991}), provides a numerically stable procedure for computing the recurrence coefficients from modified moments. 
In rational approximation, the moments arise naturally in the construction of Pad\'e approximants, and Stieltjes continued fractions~\cite{Baker1996,Brezinski1980}, whose denominators coincide with the associated orthogonal polynomials. Finally, in the classical moment problem, the positivity of all Hankel matrices is necessary and sufficient for the existence of a representing measure, and Carleman's criterion guarantees its uniqueness~\cite{Shohat1943,Simon2011}.\\

Despite their importance, the explicit determination of the moments of Laguerre-Hahn forms is a non-trivial task. The functional equation~\eqref{eq:functional} is concise, but its direct exploitation for determining the moments is hampered by the nonlinear term $B(x^{-1}u^2)$, which gives rise to Cauchy convolution products
$\sum_{\mu=0}^{n-1}(u)_\mu(u)_{n-\mu-1}$ when applied on monomials. This nonlinearity renders the algebraic structure of the moment equations substantially more complex than in the semiclassical case, for which $B = 0$ and one obtains a purely linear recurrence.\\

The goal of the present paper is to derive explicit recurrence relations for the moments of Laguerre-Hahn forms of class zero, thereby providing a recursive method for generating the moments, one by one, from the first moment.\\

Our approach is the following.
We first show that, for a Laguerre-Hahn form $u$ of arbitrary
class satisfying the functional equation~\eqref{eq:functional} with
polynomials
$$\Phi(x) = \sum_{\nu=0}^{t} c_\nu x^\nu, \quad
\psi(x)  = \sum_{\nu=0}^{p} a_\nu x^\nu, \quad
B(x)     = \sum_{\nu=0}^{q} b_\nu x^\nu,$$
the functional equation~\eqref{eq:functional} is equivalent to the
infinite algebraic system on the moments
(see Lemma~\ref{L1})
\begin{equation}\label{eq:lemma}
  -n\sum_{\nu=0}^{t} c_\nu(u)_{n-1+\nu}
  \;+\; \sum_{\nu=0}^{p} a_\nu(u)_{n+\nu}
  \;+\; \sum_{\nu=0}^{q} b_\nu
        \sum_{\mu=0}^{n+\nu-1}(u)_{\mu}(u)_{n+\nu-\mu-1}
  \;=\; 0,
  \ n \geq 0.
\end{equation}
This transforms a problem of functional analysis into a purely algebraic system on the moments. However, equation \eqref{eq:lemma} couples even and odd indices and does not directly allow the recursive computation of moments from the initial ones.

The central result of this paper, Theorem~\ref{thm:main}, goes further by restricting to class zero. By decomposing~\eqref{eq:lemma} according to the parity of $n$, the
general system reorganizes into the explicit recurrence
relations~\eqref{mom-k}-\eqref{eq:decouple-j}, in which the highest-indexed moment appears linearly: then each moment $(u)_n$ is thus entirely determined by the preceding ones, and the whole sequence is generated recursively from the first moment $(u)_0$, which, in the normalized case is
$(u)_0 = 1$.

This main result is then applied systematically to all ten canonical families ca\-talogued in the literature~\cite{Article-2-Soummi}: two analogues of the Hermite case, two analogues of the Laguerre case, four analogues of the Bessel case, and two analogues of the Jacobi case. For each family, already normalized, i.e., satisfying $(u)_0=1$, the recurrence relations generate all moments solely from the coefficients of the polynomials $\Phi(x)$, $\psi(x)$, and $B(x)$. The first moments of each family are listed explicitly; their expressions grow rapidly in algebraic complexity, making a symbolic implementation of the recurrences not only convenient but indispensable for all ten families.\\

The remainder of the paper is structured as follows. Section~\ref{Section2} introduces the necessary notation and background results on linear forms, orthogonal polynomial sequences, and Laguerre-Hahn forms.  Section~\ref{Section3} states and proves the main results of this paper. First, Lemma~\ref{L1} establishes the equivalence between the functional equation and the algebraic system~\eqref{eq:lemma} for a form of arbitrary class, which serves as the algebraic foundation of our approach. Then, Theorem~\ref{thm:main} derives, for class zero, explicit recurrence relations that allow all moments to be generated from the first moment $(u)_0$ and the coefficients of the polynomials $\Phi(x)$, $\psi(x)$, and $B(x)$ appearing in the functional equation. Section~\ref{Section4} applies the main result to each of the ten canonical families of Laguerre-Hahn forms of class zero, providing the specific recurrence relations in each case. Section~\ref{Section5} collects the explicit first-moment expressions for each family, as computed by a
{\it Mathematica$^{\circledR}$} implementation of the recurrences. The article concludes with some final remarks.

\section{Notation and basic background}\label{Section2} 

This section collects the basic definitions, notations, and results used throughout this paper.
\subsection{Basic tools}
Let $\mathcal{P}$ denote the vector space of polynomials with complex coefficients, and let $\mathcal{P'}$ be its algebraic dual space.  
The elements of $\mathcal{P'}$ will be referred to as \emph{forms} (or linear functionals).  
The pairing between $\mathcal{P}$ and $\mathcal{P'}$ is expressed through the duality brackets $\langle \cdot, \cdot \rangle$.  
For a form $u \in \mathcal{P'}$, the sequence of complex numbers $(u)_n,~ n \geq 0$, is called the \emph{moment sequence} of $u$ relative to the monomial basis $\{x^n\}_{n \geq 0}$.  
In particular, the $n$-th moment is given by 
$$
(u)_n := \langle u, x^n \rangle,
$$
so that $u$ is uniquely determined by the sequence of its moments. 

Given $u \in \mathcal{P}'$, one can introduce the formal series $F(u)(z):= \sum_{n \geq 0}(u)_n z^n$ or, equivalently, the so-called formal Stieltjes function 
\begin{equation*}
S(u)(z):=-z^{-1} F(u)\left(z^{-1}\right)=-\sum_{n \geq 0} \frac{(u)_n}{z^{n+1}}, \tag{2.1}
\end{equation*}
which provides an alternative representation of the moment sequence $\{(u)_n\}_{n \geq 0}$ \cite{Maroni-1991}.  
Since the moments determine $u$ uniquely, the function $S(u)(z)$ does so as well.\\

In the following, we shall refer to a sequence $\{P_n\}_{n \geq 0}$ as a \emph{polynomial sequence} (PS) if $\deg P_n = n$ for all $n \geq 0$. A \emph{monic polynomial sequence} (MPS) is a PS in which each polynomial has a leading coefficient equal to one.  
If ${\{P_{n}\}}_{n \geq 0}$ is a MPS, there exists a unique sequence
$\{u_n\}_{n\geq 0}$, $u_n\in\mathcal{P}^{\prime}$, called the dual sequence of $\{P_{n}\}_{n\geq 0}$, such that,
\begin{equation}\label{SucDual}
\langle u_{n},P_{m}\rangle=\delta_{n,m}, \quad n,m\geq 0.
\end{equation}
We say that a sequence of forms $\{v_n \}_{ n \geq 0}$ is normalised if and only if $\left( v_n  \right)_n = 1, $ $n \geq 0$, and if $n \geq 1$, then $\left( v_n  \right)_m = 0, $ $m=0, \ldots, n-1$. Thus, the dual sequence  $\{u_n\}_{n\geq 0}$ is normalised. The first form $u_0$ is called the canonical form of ${\{P_{n}\}}_{n \geq 0}$.

We now introduce some operations on $\mathcal{P'}$ following \cite{Maroni-1991}.  
For $c \in \mathbb{C}$, $f,p \in \mathcal{P}$, and $u \in \mathcal{P'}$, we define:
\begin{eqnarray}
&&\langle fu, p \rangle = \langle u, fp \rangle, \quad
\langle u', p \rangle = -\langle u, p' \rangle, \label{op_2} \\
&&\langle (x-c)^{-1}u, p \rangle = \langle u, \theta_c p \rangle
= \left\langle u, \frac{p(x)-p(c)}{x-c} \right\rangle. \label{op_3} 
\end{eqnarray}
Given $f \in \mathcal{P}$ and $u \in \mathcal{P'}$, the product $uf$ is defined by  $(u f)(x) := \left\langle u, \displaystyle\frac{x f(x)-\zeta f(\zeta)}{x-\zeta} \right\rangle $.

This definition allows us to introduce the \emph{Cauchy product} of two forms $u,v \in \mathcal{P'}$ by 
\begin{equation}\label{op4}
\langle uv, f \rangle := \langle u, v f \rangle, \quad f \in \mathcal{P}.
\end{equation}

A linear functional $u$ is called \emph{regular} (or \emph{quasi-definite}) if there exists a sequence of polynomials  $\{P_n\}_{n \geq 0}$ such that \cite{Chihara-1978}
\begin{equation}\label{regular}
\langle u, P_n P_m \rangle = r_n \, \delta_{n,m}, \quad n, m \geq 0,
\end{equation}
where $\{r_n\}_{n \geq 0}$ is a sequence of nonzero complex numbers and $\delta_{n,m}$ denotes the Kronecker symbol.
The sequence $\{P_n\}_{n\geq 0}$ is then said orthogonal with respect to $u$. 
Then, necessarily, $\{P_n\}_{n \geq 0}$ is a PS, $u=(u)_0u_0$, and $\{P_n\}_{n \geq 0}$ and $u$ can be normalized. In the sequel, we shall consider that each $P_n(x)$ is monic, and $(u)_0=1$ (i.e. $u=u_0$).
Henceforth, a monic orthogonal polynomial sequence $\{P_n\}_{n\geq 0}$ will be indicated as MOPS. 

It is well known that a MOPS is characterized by the following second-order linear recurrence relation and initial conditions, respectively \cite{Chihara-1978}
\begin{eqnarray}
&& P_0(x)=1,\quad  P_1(x)=x-\beta_0, \label{ic_TTRR}\\
&& P_{n+2}(x)=(x-\beta_{n+1})P_{n+1}(x)-\gamma_{n+1}P_{n}(x),~~n\geq 0, \label{TTRR}
\end{eqnarray}
being $\{\beta_n\}_{n\geq 0}$ and $\{\gamma_{n+1}\}_{n\geq 0}$ sequences of complex numbers such that $\gamma_{n+1}\neq0$ for all $ n\geq 0$, that we designate by recurrence coefficients.

A linear functional $u$ is regular if and only if $H_n(u) \neq 0$ for all $n \geq 0$, where $H_n(u)$ is the Hankel determinant  defined by \cite{Chihara-1978, Maroni-1991}
\[
H_n(u) = \det \left( (u)_{i+j} \right)_{i, j = 0}^n =
\begin{vmatrix}
(u)_0 & (u)_1 & \cdots & (u)_n \\
(u)_1 & (u)_2 & \cdots & (u)_{n+1} \\
\vdots & \vdots & \ddots & \vdots \\
(u)_n & (u)_{n+1} & \cdots & (u)_{2n}
\end{vmatrix}, \ n\geq 0.
\]
In this case, the corresponding MOPS $\{P_n\}_{n \geq 0}$ is given by $P_0(x)=1$, and
\[
P_n(x) = \frac{(-1)^n}{H_{n-1}(u)}
\begin{vmatrix}
1 & x & \cdots & x^n \\
(u)_0 & (u)_1 & \cdots & (u)_n \\
(u)_1 & (u)_2 & \cdots & (u)_{n+1} \\
\vdots & \vdots & \ddots & \vdots \\
(u)_{n-1} & (u)_n & \cdots & (u)_{2n-1}
\end{vmatrix}, \ n\geq 1,
\]
which show that $\{P_n(x)\}_{n\geq 0}$ can be obtained from $\{(u)_n\}_{n\geq 0}$.

Let  $\{P_n^{(1)}\}_{n\geq0}$ be the associated polynomial sequence of order one of the MPS  $\{ P_n\}_{n\geq0}$ with respect to the canonical form $u=u_0$. It is well known that \cite{Chihara-1978}
$$
 P_n^{(1)}(x):=(u\theta_0P_{n+1})(x)=\bigg\langle  u,\frac{P_{n+1}(x)- P_{n+1}(\xi)}{x-\xi}  \bigg\rangle.
$$
The Stieltjes function of $u^{(1)}$ is expressed in terms of that of $u$ as  \cite{Maroni-1991} 
$$
\gamma_1 S\left(u^{(1)}\right)(z)=-\frac{1}{S(u)(z)}-\left(z-\beta_0\right).
$$

\subsection{Laguerre-Hahn forms}

\begin{definition} \cite{Dzoumba-these-1985,Magnus-1983,Maroni-1983}
A regular form $u$, with $(u)_0=1$, is said to be a Laguerre-Hahn form if its formal Stieltjes function satisfies the Riccati equation
\begin{equation}\label{Riccati}
A(z)S'(u)(z)=B(z)S^2(u)(z)+C(z)S(u)(z)+D(z),
\end{equation}
where $A$, $B$, $C$, and $D$ are polynomials.\\
The sequence \(\{P_n\}_{n\geq 0}\) orthogonal with respect to \( u \) is also called a Laguerre-Hahn sequence. 
\end{definition}

\begin{remark} \cite{Maroni-1991}
If $A=0$ identically, the form $u$ is classified as a second-degree form \cite{Maroni-1995}. 
If $A$ is not identically zero, it may be assumed, without loss of generality, that it is monic; and we let $A:=\Phi$. Under this normalization, the condition $B \neq 0$ characterizes $u$ as a strict Laguerre-Hahn form, whereas the case $B = 0$ corresponds to a semiclassical form.
\end{remark}

There are several characterizations of Laguerre-Hahn forms. Some of them are listed in the following result.
\begin{proposition} \cite{Alaya-Maroni,Dini-these-1988,Bouakkaz-Maroni-1991,Maroni-1991} 
Let $u$ be a regular and normalized form, i.e.,
 $(u)_0=1$, and let $\{P_n\}_{n\geq 0}$ be its corresponding MOPS. The following statements are equivalent
\begin{enumerate}
\item[(i)] $u$ is a Laguerre-Hahn form satisfying \eqref{Riccati} with $A=\Phi$.
\item[(ii)]  \cite{Dini-these-1988} $u$ satisfies the functional equation
\begin{equation}\label{Laguerre-Hahn-EF}
(\Phi u)'+\psi u+B(x^{-1}u^2)=0,
\end{equation}
where $\Phi$, $B$, $C$, and $D$ are the polynomials in \eqref{Riccati} and
\begin{align*}
&C=-\Phi'-\psi,\\
&D=-(u\theta_0\Phi)'-(u\theta_0\psi)-(u^2\theta_0^2B).
\end{align*}
\item[(iii)] \cite{Dini-these-1988} Each polynomial \(P_n, n \geq 0\), verifies the so-called structure relation
\begin{equation}\label{Dini_St_Rel}
\Phi(x)P_{n+1}^{\prime}(x) - B(x)P_n^{(1)}(x) = \sum_{\mu=n-s}^{n+d} \theta_{n,\mu} P_\mu(x), \quad n \geq s + 1,
\end{equation}
where \(\Phi\) and \(B\) are the polynomials defined in (i),  \(\{P_n^{(1)}\}_{n \geq 0}\) is the sequence of associated orthogonal polynomials of order 1 of \(\{P_n\}_{n \geq 0}\),  
$d=\max(t,q)$, $s=\max(p-1,d-2)$, being $t$, $p$, and $q$ the degrees of  $\Phi$, $\psi$, and $B$, respectively. 
\end{enumerate}
\end{proposition}

It is worth noting that the above functional equation \eqref{Laguerre-Hahn-EF} is not uniquely determined.
Indeed, if $u$ is a Laguerre-Hahn form and $\chi$ is an arbitrary polynomial, then $u$ also satisfies
\[
(\chi \Phi u)' + \bigl(\chi \psi - \chi' \Phi\bigr) u
+ (\chi B)\,(x^{-1}u^{2}) = 0.
\]
This observation motivates the following definition.
\begin{definition} \cite{Alaya-Maroni, Bouakkaz-Maroni-1991}
The class of a Laguerre-Hahn form $u$ is the non-negative integer number defined as
$$
s:=\min\max\big\{\deg{\psi}-1, \max\{\deg{\Phi}, \deg{B}\}-2\big\},
$$
where the minimum is taken among all polynomials $\Phi, \psi$ and $B$ such that $u$
satisfies \eqref{Laguerre-Hahn-EF}.
\end{definition}

Taking into account that the class of a Laguerre-Hahn form is crucial to state a hierarchy of such families, we need to give a criterion to characterize it.
\begin{proposition}\label{proposition-simplification} \cite{Alaya-Maroni,Bouakkaz-Maroni-1991}
Let $u$ be a Laguerre-Hahn form and let $\Phi$ and $\psi$ be non-zero polynomials 
such that \eqref{Laguerre-Hahn-EF} holds.
 Let
 \begin{equation}\label{s_class_LH}
 s=\max\big\{\deg{\psi}-1, \max\{\deg{\Phi}, \deg{B}\}-2\big\}.
 \end{equation} 
 Then $s$ is the class of $u$ if and only if
\begin{equation*}
\prod_{c\in\mathcal{Z}_{\Phi}}{\Big(|\Phi'(c)+\psi(c)|+|B(c)|+|\langle u, \theta_c^2\Phi+\theta_c\psi+u\theta_0\theta_c B\rangle|\Big)}\neq0,
\end{equation*}
where $\mathcal{Z}_{\Phi}$ denotes the set of zeros of $\Phi$.
\end{proposition}

Let $a \in \mathbb{C}\setminus\{0\}$ and $b \in \mathbb{C}$. 
If a Laguerre-Hahn form $u$ of class $s$ satisfies \eqref{Laguerre-Hahn-EF}, 
then the shifted form $\tilde{u} = (h_{a^{-1}} \circ \tau_{-b})u$ 
is also a Laguerre-Hahn form of class $s$ and satisfies \cite{Bouakkaz-these,Bouakkaz-Maroni-1991,Dini-these-1988}
\begin{equation*}
(\tilde{\Phi}\,\tilde{u})' + \tilde{\psi}\,\tilde{u}
+ \tilde{B}\bigl(x^{-1}\tilde{u}^2\bigr) = 0,
\end{equation*}
where 
$$
\tilde{\Phi}(x) = a^{-\deg\Phi}\,\Phi(ax+b),\ 
\tilde{\psi}(x) = a^{1-\deg\Phi}\,\psi(ax+b),\ 
\tilde{B}(x) = a^{-\deg\Phi}B(ax+b).
$$
Hence, a displacement does not change either the Laguerre-Hahn character or the class of a Laguerre-Hahn linear functional. Therefore, we may consider canonical functional equations by re-situating the zeros of $\Phi$ in 
\eqref{Laguerre-Hahn-EF}.

\section{Recurrence relations for the moments}\label{Section3}

In this section, we provide a general recurrence relation for the moments of any Laguerre-Hahn form of class zero, thereby allowing the generation of all moments from the first one $(u)_0 = 1$.

Throughout this work, we use the conventions
\[
\sum_{a}^{b} = 0, \  \text{whenever } b < a,
\quad\text{and}\quad
(u)_k = 0, \  \text{for all } k < 0.
\]

First, we present a general lemma, with no assumption on the class, stating the equivalence between the functional equation \eqref{Laguerre-Hahn-EF} and an algebraic system satisfied by the moments. This equivalence constitutes a crucial
theoretical foundation. After that, we restrict to class zero and show how this system reorganises into explicit recurrences.

\begin{lemma}\label{L1}
Let $u$ be a regular linear form defined on the dual space ${\mathcal  P}'$. The following statements are equivalent:
\begin{enumerate}
\item[(a)] $u$ is a Laguerre-Hahn form satisfying the following  functional equation
\begin{equation}\label{(4.1.2.1)}
(\Phi u)' + \psi u + B(x^{-1}u^2) = 0,
\end{equation}
where
\begin{equation}\label{polcoeff_funEq}
\Phi(x) = \sum_{\nu=0}^{t} c_\nu x^\nu, 
\quad 
\psi(x) = \sum_{\nu=0}^{p} a_\nu x^\nu, 
\quad 
B(x) = \sum_{\nu=0}^{q} b_\nu x^\nu.
\end{equation}
\item[(b)] The moments of the form $u$ satisfy the following recurrence relation:
\begin{equation}
-n \sum_{\nu=0}^{t} c_\nu (u)_{n-1+\nu}
 + \sum_{\nu=0}^{p} a_\nu (u)_{n+\nu}
 + \sum_{\nu=0}^{q} b_\nu \sum_{\mu=0}^{\,n+\nu-1} (u)_\mu (u)_{n+\nu-\mu-1}
 = 0, \  n \geq 0.
\end{equation}
\end{enumerate}
\end{lemma}
\begin{proof}
The proof of this lemma requires the operations \eqref{op_2}-\eqref{op_3} defined on ${\mathcal  P}'$; it is  straightforward from the fact that Equation \eqref{(4.1.2.1)} is equivalent to $\langle (\Phi u)' + \psi u + B(x^{-1} u^2), x^n \rangle = 0,~ n \geq 0$,
and is therefore omitted.
\end{proof}

In the sequel, we assume that $\left\{P_n\right\}_{n \geq 0}$ is a Laguerre-Hahn sequence of class zero satisfying \eqref{ic_TTRR}-\eqref{TTRR} and its corresponding form $u$ satisfying \eqref{Laguerre-Hahn-EF} with
\[ 
\Phi(x) = c_2 x^2 + c_1 x + c_0, \quad \psi(x) = a_1 x + a_0, \quad B(x) = b_2 x^2 + b_1 x + b_0,
\]
with 
\[
|c_2| + |a_1| + |b_2| > 0.
\]

Before stating the main result, we recall the following useful
identity on convolution sums, which plays a central role in its proof.
 
\begin{lemma}\label{lem:symsum}
Let $N\geq0$ be an integer and $(U_k)_{k\ge0}$ be a sequence. Then,
\[
\sum_{\mu=0}^{N}U_\mu U_{N-\mu}
=
\begin{cases}
\displaystyle 2\sum_{\mu=0}^{M-1}U_\mu U_{2M-\mu}+U_M^2, & \text{if } N=2M \text{ is even},\\[8pt]
\displaystyle 2\sum_{\mu=0}^{M}U_\mu U_{2M+1-\mu}, & \text{if } N=2M+1 \text{ is odd}.
\end{cases}
\]
\end{lemma}

We are now in a position to state the main result of the paper.

\begin{theorem}\label{thm:main}
Let $u$ be a regular linear form defined on the dual space ${\mathcal  P}'$. The following statements are equivalent:
\begin{enumerate}
\item[(a)] $u$ is a Laguerre–Hahn form of class zero satisfying \eqref{(4.1.2.1)}-\eqref{polcoeff_funEq}.
\vspace{0.15cm}
\item[(b)] The moments of the form $u$ satisfy the following recurrence relations \\
\begin{enumerate}
\item[$\bullet$] for $k \in \{1,2,3,4\}$,   
\begin{equation}\label{mom-k}
(u)_k = \frac{M_k}{E_{0,\,k-3}},    
\end{equation}
\item[$\bullet$] for $n \geq 1$ and $j \in \{0,1\}$, 
\begin{equation}\label{eq:decouple-j}
\begin{aligned}
& E_{n,j} (u)_{2n+3+j} + F_{n,j} (u)_{2n+2+j} + G_{n,j} (u)_{2n+1+j} \\
&\quad + b_2 \Bigl[ 2\sum_{\nu=3}^{n+1} (u)_\nu (u)_{2n+3+j-\nu} 
    + \delta_{j,1} (u)_{n+2}^2 \Bigr] \\
&\quad + b_1 \Bigl[ 2\sum_{\nu=2}^{n} (u)_\nu (u)_{2n+2+j-\nu} 
    + \delta_{j,0} (u)_{n+1}^2 
    + \delta_{j,1}\cdot 2\,(u)_{n+1}(u)_{n+2} \Bigr] \\
&\quad + b_0 \Bigl[ 2\sum_{\nu=1}^{n} (u)_\nu (u)_{2n+1+j-\nu} 
    + \delta_{j,1} (u)_{n+1}^2 \Bigr] = 0,\quad n\geq 1,\quad j\in\{0,1\},
\end{aligned}
\end{equation}
\end{enumerate}
where
\begin{align}
E_{n,j} &:= a_1 - (2n+2+j)c_2 + 2b_2 (u)_0 \neq 0, \\
F_{n,j} &:= a_0 - (2n+2+j)c_1 + 2b_1 (u)_0 + 2b_2 (u)_1, \\
G_{n,j} &:= -(2n+2+j)c_0 + 2b_0 (u)_0 + 2b_1 (u)_1 + 2b_2 (u)_2,
\end{align}
and
\begin{align}
M_1 :=& -\bigl[a_0 (u)_0 + b_1 (u)_0^2\bigr], \\
M_2 :=& c_0 (u)_0 - (a_0 - c_1)(u)_1 - b_0 (u)_0^2 - 2b_1 (u)_0 (u)_1 - b_2 (u)_1^2, \\
M_3 :=& 2c_0 (u)_1 + 2c_1 (u)_2 - a_0 (u)_2 - 2b_0 (u)_0 (u)_1 \nonumber\\
& - 2b_1 (u)_0 (u)_2 - b_1 (u)_1^2 - 2b_2 (u)_1 (u)_2, \\
M_4 :=& 3c_0 (u)_2 + 3c_1 (u)_3 - a_0 (u)_3 - 2b_0 (u)_0 (u)_2 - b_0 (u)_1^2 \nonumber\\
& - 2b_1 (u)_0 (u)_3 - 2b_1 (u)_1 (u)_2 - 2b_2 (u)_1 (u)_3 - b_2 (u)_2^2.
\end{align}
\end{enumerate}
\end{theorem}

\begin{proof}
By Lemma~\ref{L1} with $t=2$, $p=1$, $q=2$, the form $u$ is a Laguerre–Hahn of class zero if and only if
\begin{align}
 -n c_0 (u)_{n-1} - n c_1 (u)_n - n c_2 (u)_{n+1} + a_0 (u)_n + a_1 (u)_{n+1}\nonumber\\
 + \sum_{k=0}^{2} b_{k} \sum_{\mu=0}^{n+k-1}(u)_\mu (u)_{n+k-\mu-1}=0,\quad n\ge0, \label{eq:gen}
\end{align}
with $|c_2| + |a_1| + |b_2| > 0$.\\
Let us define
\[
S_k(n):=\sum_{\mu=0}^{n+k-1}(u)_\mu (u)_{n+k-\mu-1},\quad k=0,1,2,
\]
so that \eqref{eq:gen} can be written as
\begin{align*}
K(n) := & -n c_0 (u)_{n-1} - n c_1 (u)_n - n c_2 (u)_{n+1} + a_0 (u)_n + a_1 (u)_{n+1}\\
&+ b_0 S_0(n) + b_1 S_1(n) + b_2 S_2(n)=0, \quad n\geq0,
\end{align*}
i.e., 
$$
K(k)=0, \quad k\in\{0, 1, 2, 3\}
$$
and
$$K(2n+3)=0,\quad K(2n+2)=0, \quad n\geq1.
$$
It is straightforward to verify that, for $k=0, 1, 2, 3$, the condition $K(k)=0$ yields, respectively,
\begin{align*}
\big[a_1 + 2b_2 (u)_0\big](u)_1 =& -a_0 (u)_0 - b_1 (u)_0^2,\\
\big[a_1 - c_2 + 2b_2 (u)_0\big](u)_2 =& c_0 (u)_0 - (a_0 - c_1)(u)_1 \\
& - b_0 (u)_0^2 - 2b_1 (u)_0 (u)_1 - b_2 (u)_1^2,\\
\big[a_1 -2c_2 + 2b_2 (u)_0 \big] (u)_3 =& 2c_0 (u)_1 + 2c_1 (u)_2 - a_0 (u)_2 - 2b_0 (u)_0 (u)_1 \notag\\
& - 2b_1 (u)_0 (u)_2 - b_1 (u)_1^2 - 2b_2 (u)_1 (u)_2,\\  
\big[ -3c_2 + a_1 + 2b_2 (u)_0 \big] (u)_4 =& 3c_0 (u)_2 + 3c_1 (u)_3 - a_0 (u)_3 \notag\\
& - 2b_0 (u)_0 (u)_2 - b_0 (u)_1^2 - 2b_1 (u)_0 (u)_3 \notag\\
& - 2b_1 (u)_1 (u)_2 - 2b_2 (u)_1 (u)_3 - b_2 (u)_2^2,
\end{align*}
which are exactly the cases $k=0,1,2,$ and $3$ of equation \eqref{mom-k}.\\
It remains to verify that the conditions $K(2n+3)=0$ and $K(2n+2)=0$, for $n\geq 1$, are equivalent to equation \eqref{eq:decouple-j} with $j=0$ and $j=1$, respectively.

Applying Lemma~\ref{lem:symsum} to each sum yields
\[
\begin{aligned}
S_0(2n+2)=&\sum_{\mu=0}^{2n+1}(u)_\mu (u)_{2n+1-\mu}
        =2\sum_{\mu=0}^{n}(u)_\mu (u)_{2n+1-\mu} \\
        =&2(u)_0(u)_{2n+1}+2\mathcal{A}_n^{(2)}, \ n\geq1,\\[2pt]
S_1(2n+2)=&\sum_{\mu=0}^{2n+2}(u)_\mu (u)_{2n+2-\mu}
        =2\sum_{\mu=0}^{n}(u)_\mu (u)_{2n+2-\mu}+(u)_{n+1}^2\\
        =&2(u)_0(u)_{2n+2}+2(u)_1(u)_{2n+1}+2\mathcal{B}_n^{(2)}+(u)_{n+1}^2,\ n\geq1,\\[2pt]
S_2(2n+2)=&\sum_{\mu=0}^{2n+3}(u)_\mu (u)_{2n+3-\mu}
        =2\sum_{\mu=0}^{n+1}(u)_\mu (u)_{2n+3-\mu}\\
    =&2(u)_0(u)_{2n+3}+2(u)_1(u)_{2n+2}+2(u)_2(u)_{2n+1}+2\mathcal{C}_n^{(2)}, \ n\geq1,
\end{aligned}
\]
where
$$\mathcal{A}_n^{(2)}:=\sum_{\nu=1}^{n}(u)_\nu (u)_{2n+1-\nu},\  \mathcal{B}_n^{(2)}:=\sum_{\nu=2}^{n}(u)_\nu (u)_{2n+2-\nu},\  \mathcal{C}_n^{(2)}:=\sum_{\nu=3}^{n+1}(u)_\nu (u)_{2n+3-\nu},\ n\geq1.
$$
Substituting these expressions for $S_0$, $S_1$, and $S_2$ into $K(2n+2)$ yields
\[
\begin{aligned}
K(2n+2)=&-\,(2n+2)c_0(u)_{2n+1}-(2n+2)c_1(u)_{2n+2}-(2n+2)c_2(u)_{2n+3}\\
&+a_0(u)_{2n+2}+a_1(u)_{2n+3}\\
&+b_0\bigl[2(u)_0(u)_{2n+1}+2\mathcal{A}_n^{(2)}\bigr]\\
&+b_1\bigl[2(u)_0(u)_{2n+2}+2(u)_1(u)_{2n+1}+2\mathcal{B}_n^{(2)}+(u)_{n+1}^2\bigr]\\
&+b_2\bigl[2(u)_0(u)_{2n+3}+2(u)_1(u)_{2n+2}+2(u)_2(u)_{2n+1}+2\mathcal{C}_n^{(2)}\bigr], \ n\geq1.
\end{aligned}
\]
According to the latter, and taking into account the fact that $K(2n+2)=0$, we obtain exactly the relation \eqref{eq:decouple-j} with $j=0$.

\medskip
Now, using again Lemma~\ref{lem:symsum},
\begin{align*}
S_0(2n+3)&=\sum_{\mu=0}^{2n+2}(u)_\mu (u)_{2n+2-\mu}
        =2\sum_{\mu=0}^{n}(u)_\mu (u)_{2n+2-\mu}+(u)_{n+1}^2\\
        &=2(u)_0(u)_{2n+2}+2\mathcal{A}_n^{(3)}+(u)_{n+1}^2, \ n\geq1,\\[2pt]
S_1(2n+3)&=\sum_{\mu=0}^{2n+3}(u)_\mu (u)_{2n+3-\mu}
        =2\sum_{\mu=0}^{n+1}(u)_\mu (u)_{2n+3-\mu}\\
        &=2(u)_0(u)_{2n+3}+2(u)_1(u)_{2n+2}+2 \mathcal{B}_n^{(3)}, \ n\geq1,\\[2pt]
S_2(2n+3)&=\sum_{\mu=0}^{2n+4}(u)_\mu (u)_{2n+4-\mu}=2\sum_{\mu=0}^{n+1}(u)_\mu (u)_{2n+4-\mu}+(u)_{n+2}^2\\
&=2(u)_0(u)_{2n+4}+2(u)_1(u)_{2n+3}+2(u)_2(u)_{2n+2}+2\mathcal{C}_n^{(3)}+(u)_{n+2}^2, \ n\geq1,
\end{align*}
where
$$
\mathcal{A}_n^{(3)}:=\sum_{\nu=1}^{n}(u)_\nu (u)_{2n+2-\nu},\ 
\mathcal{B}_n^{(3)}:=\sum_{\nu=2}^{n+1}(u)_\nu (u)_{2n+3-\nu},\
\mathcal{C}_n^{(3)}:=\sum_{\nu=3}^{n+1}(u)_\nu (u)_{2n+4-\nu},\ n\geq1.
$$
Substituting these expressions for $S_0,S_1,S_2$ into $K(2n+3)$ yields
\[
\begin{aligned}
K(2n+3)=&-\,(2n+3)c_0(u)_{2n+2}-(2n+3)c_1(u)_{2n+3}-(2n+3)c_2(u)_{2n+4}\\
&+a_0(u)_{2n+3}+a_1(u)_{2n+4}\\
&+b_0\bigl[2(u)_0(u)_{2n+2}+2\mathcal{A}_n^{(3)}+(u)_{n+1}^2\bigr]\\
&+b_1\bigl[2(u)_0(u)_{2n+3}+2(u)_1(u)_{2n+2}+2\mathcal{B}_n^{(3)}\bigr]\\
&+b_2\bigl[2(u)_0(u)_{2n+4}+2(u)_1(u)_{2n+3}+2(u)_2(u)_{2n+2}+2\mathcal{C}_n^{(3)}+(u)_{n+2}^2\bigr], \ n\geq1.
\end{aligned}
\]
Based on the latter and the condition $K(2n+3)=0$, we recover exactly the relation \eqref{eq:decouple-j} with $j=1$, which completes the proof of the
theorem.
\end{proof}

\begin{remark}
For a Laguerre form of class zero, the recurrence coefficients $\beta_0$ and $\gamma_1$ are directly expressed in terms of $(u)_0$ and the coefficients of $\Phi(x)$, $\psi(x)$, and $B(x)$ in the functional equation \eqref{(4.1.2.1)}. Indeed, using the second-order linear recurrence relation  \eqref{ic_TTRR}-\eqref{TTRR}, and the regularity condition \eqref{regular}, we obtain 
$$
\beta_0 = \frac{(u)_1}{(u)_0},
$$
and
$$
\gamma_1 = \frac{(u)_2 (u)_0 - (u)_1^2}{(u)_0^2}.
$$
From the expressions of $(u)_1$ and $(u)_2$ in \eqref{mom-k}, we easily deduce that
$$
\beta_0 = \frac{M_1}{(u)_0 E_{0,-2}}= \frac{-a_0(u)_0 - b_1(u)_0^2}{(u)_0 \bigl(a_1 + 2b_2(u)_0\bigr)},
$$
and
$$
\gamma_1 = \frac{ (u)_0 E_{0,-2}^2 \bigl(c_0 (u)_0 - b_0 (u)_0^2\bigr) - (u)_0 E_{0,-2} \bigl(a_0 - c_1 + 2b_1 (u)_0\bigr) M_1 - M_1^2 \bigl( (u)_0 b_2 + E_{0,-1} \bigr) }{ (u)_0^2 \, E_{0,-1} \, E_{0,-2}^2 },
$$
where $M_1 = -a_0 (u)_0 - b_1 (u)_0^2$, $E_{0,-2} = a_1 + 2b_2 (u)_0$, and $E_{0,-1} = a_1 - c_2 + 2b_2 (u)_0$.
\end{remark}

\section{Moments recurrences for the ten canonical Laguerre-Hahn families of class zero}\label{Section4}

In this section, Theorem~\ref{thm:main} is applied to each of the ten canonical Laguerre-Hahn families of class zero classified in~\cite{Article-2-Soummi}, providing for each one the explicit recurrence relations satisfied by its moments.

For each canonical family, it suffices to compute the explicit formulas of $M_k$, $E_{n,j}$, $F_{n,j}$, and $G_{n,j}$, which depend only on the coefficients $c_i$, $a_i$, and $b_i$ (see \eqref{polcoeff_funEq}) of the polynomials $\Phi(x)$, $\psi(x)$ and $B(x)$ appearing in the funcional equation \eqref{(4.1.2.1)}, and the first moment $(u)_0=1$. This procedure is applied to all ten families in the following sections.

For each family, we begin by specifying the regularity conditions and the polynomials  $\Phi(x)$, $\psi(x)$, and $B(x)$. Next, we present the formulas of the moments recurrences $M_k$, $k=1,2,3,4$, $E_{n,j}$, $F_{n,j}$, and $G_{n,j}$ obtained with {\it Mathematica$^{\circledR}$}.

\subsection{Two cases analogous to the Hermite case}

\subsubsection{Case 1 analogous to Hermite}

\noindent {\bf Regularity conditions}
$$\lambda, \rho, \tau \in \mathbb{C},\quad \rho \neq 0,\quad \tau\neq -n, \quad n \geq 1.$$

\noindent {\bf Functional equation}
\begin{eqnarray}
&&\Phi(x) = 1, \quad
\psi(x) = 2\displaystyle\frac{2-\rho}{\rho}x - \frac{4\lambda}{\rho},\notag\\
&&B(x) = 2\displaystyle\frac{\rho-1}{\rho}x^2 + 2\lambda \displaystyle\frac{2-\rho}{\rho}x + 1 - \rho(\tau+1) - \displaystyle\frac{2\lambda^2}{\rho}.\notag
\end{eqnarray}

\noindent {\bf Moments recurrence coefficients}
\begin{eqnarray}
&& M_1 = 2 \lambda, \quad M_2  =2 \lambda^2+\rho(1+ \tau), \quad
M_3  = 2 \lambda ^3+2 \lambda  \rho  (\tau +1),\notag \\
&&M_4 = 2 \lambda ^4+3 \lambda ^2 \rho  (\tau +1)+\frac{1}{2} \rho  (\tau +1) (\rho  (\tau +1)+\tau +2).\notag\\
&&E_{n, j} =2, \quad
F_{n, j}=0, \quad
G_{n, j}=-j-2(1+n+\tau).\notag
\end{eqnarray}

\subsubsection{Case 2 analogous to Hermite}
\noindent {\bf Regularity conditions}
$$\lambda, \rho \in \mathbb{C},\quad \rho \neq 0.$$

\noindent {\bf Functional equation}
$$\Phi(x) = 1, \quad \psi(x) = -2x, \quad
B(x) = 2x^2-2\lambda x+1-\rho.$$

\noindent {\bf Moments recurrence coefficients}
\begin{eqnarray}
&& M_1= 2 \lambda,\quad
M_2=2 \lambda ^2+\rho,\quad
M_3 = 2 \lambda  \left(\lambda ^2+\rho \right) ,\notag\\
&&M_4 =2 \lambda ^4+3 \lambda ^2 \rho +\frac{1}{2} \rho  (\rho +1).\notag\\
&&
E_{n, j}=2, \quad
F_{n, j}=0, \quad
G_{n, j}=-j-2n. \notag
\end{eqnarray}

\subsection{Two cases analogous to the Laguerre case}

\subsubsection{Case 1 analogous to Laguerre}
\noindent {\bf Regularity conditions}
$$\lambda, \rho, \tau, \alpha\in\mathbb{C}, \quad \rho\neq 0, \quad \alpha+\tau\neq-(n+1),\quad \tau\neq-(n+1), \quad n\geq 0.$$

\noindent {\bf Functional equation}
\begin{eqnarray}
&&\Phi(x)=x,\quad
\psi(x)=\displaystyle\frac{2-\rho}{\rho}x+\frac{\rho-2}{\rho}(2\tau+\alpha+1)-\frac{2\lambda}{\rho}, \notag \\
&&B(x)=\displaystyle\frac{\rho-1}{\rho}x^2+\left\{2\displaystyle\frac{1-\rho}{\rho}(2\tau+\alpha+1)+\lambda\displaystyle\frac{2-\rho}{\rho}\right\} x \notag \\
&&\quad\quad\quad~ +(2\tau+\alpha+1+\lambda)\left\{\displaystyle\frac{\rho-1}{\rho}(2\tau+\alpha+1)+1-\frac{\lambda}{\rho}\right\} \notag \\
&&\quad\quad\quad~ -\rho(\tau+1)(\tau+\alpha+1). \notag 
\end{eqnarray}

\noindent {\bf Moments recurrence coefficients}
\begin{eqnarray}
M_1& = & \alpha +\lambda +2 \tau +1,\notag\\
M_2 & = & \rho  (\tau +1) (\alpha +\tau +1)+(\alpha +\lambda +2 \tau +1)^2 ,\notag\\
 M_3 & = & \rho  (\tau +1) (\alpha +\tau +1) (3 \alpha +2 \lambda +6 \tau +5)+(\alpha +\lambda +2 \tau
   +1)^3 ,\notag\\
 M_4& = &  \rho ^2 (\tau +1)^2 (\alpha +\tau +1)^2+\rho  (\tau +1) (\alpha +\tau +1)\notag\\
&& \left(3 \lambda ^2+4 \lambda  (2 \alpha +4 \tau +3)+25 \alpha  \tau +22 \alpha + 6 \alpha ^2+25 \tau
   ^2+44 \tau +2\right)\notag\\
&&+(\alpha+\lambda +2 \tau +1)^4 .\notag\\
E_{n, j} & = &1, \quad
F_{n, j}=-3-j-2 n-\alpha-2 \tau, \quad
G_{n, j}=2 \lambda -2 \tau  (\alpha +\tau ).\notag
\end{eqnarray}

\subsubsection{Case 2 analogous to Laguerre}
\noindent {\bf Regularity conditions}
$$\lambda, \rho, \alpha \in \mathbb{C},\quad \rho \neq 0,\quad \alpha \neq -n, \quad n \geq 1.$$

\noindent {\bf Functional equation}
\begin{eqnarray}
&&\Phi(x) = x, \quad
\psi(x) = -x + \alpha - 1, \notag \\
&&B(x) = x^2 + \big\{2(1-\alpha) - \lambda\big\}x + \alpha(\alpha - 1 + \lambda) - \rho. \notag
\end{eqnarray}

\noindent {\bf Moments recurrence coefficients}
\begin{eqnarray}
M_1 & = & \alpha +\lambda -1 ,\notag\\
M_2 & = & \lambda ^2+2 (\alpha -1) \lambda+\alpha ^2 -2 \alpha +\rho +1 ,\notag\\
M_3 & = & \lambda ^3+3 (\alpha -1) \lambda ^2+\lambda  \left(3 \alpha ^2-6 \alpha +2 \rho +3\right)+
(3 \alpha -1) \rho +(\alpha -1)^3 ,\notag\\
M_4 & = &\lambda ^4+ 4 (\alpha -1) \lambda^3+3 \lambda ^2 \left(2 \alpha ^2-4 \alpha +\rho +2\right)\notag\\
&&+4 \lambda  \left(\alpha ^3-3 \alpha ^2+2 \alpha  \rho +3 \alpha -\rho -1\right)
+\rho ^2+3 \left(2 \alpha ^2-\alpha +1\right) \rho +(\alpha -1)^4 .\notag\\
E_{n, j}& = &1, \quad
F_{n, j}=-1-j-2 n-\alpha, \quad
G_{n, j}=2(-1+\alpha+\lambda).\notag
\end{eqnarray}

\subsection{Four cases analogous to the Bessel case}

\subsubsection{Case 1 analogous to Bessel}
\noindent {\bf Regularity conditions}
$$\lambda, \rho, \tau, \alpha \in \mathbb{C}, \quad
\rho \neq 0, \quad \tau+1 \neq -n, \quad \tau+2\alpha-1 \neq -n, \quad \tau+\alpha \neq -n/2, \quad n\geq 0.$$

\noindent {\bf Functional equation}
\begin{eqnarray}
&&\Phi(x)=x^2 \notag \\
&&\psi(x)=2\displaystyle\left\{\frac{1-\rho}{\rho}+\frac{\rho-2}{\rho}(\tau+\alpha)\right\}x+\frac{2}{\rho}(2(\tau+\alpha)-1)\beta_0-2\frac{1-\alpha}{\tau+\alpha}, \notag \\
&&B(x)=\displaystyle\frac{1-\rho}{\rho}\left\{2(\tau+\alpha)-1\right\}x^2+2\left\{\frac{1-\alpha}{\tau+\alpha}+\beta_0\left[(\tau+\alpha)\frac{\rho-2}{\rho}+\frac{1}{\rho}\right]\right\}x \notag \\
&&\quad\quad\quad~~ +\displaystyle\left\{2(\tau+\alpha)\frac{\rho-1}{\rho}+\frac{1}{\rho}\right\}\beta_0^2-b_1\beta_0-\rho\frac{(\tau+1)(\tau+2\alpha-1)}{(2\tau+2\alpha-1)(\tau+\alpha)^2}. \notag
\end{eqnarray}

\noindent {\bf Moments recurrence coefficients}
\begin{eqnarray}
 M_1& = & \frac{2 (\alpha -1)}{\alpha +\tau -1}-2 \lambda  (\alpha +\tau ),\notag\\
 M_2& = & \frac{\rho  (\tau +1) (2 \alpha +\tau -1)}{(\alpha +\tau )^2 (2 \alpha +2 \tau -1)}
 -\frac{(2 \alpha +2 \tau+1) \left(\lambda  (\alpha +\tau -1) (\alpha +\tau )-\alpha +1\right)^2}{(\alpha +\tau -1)^2 (\alpha +\tau )^2},\notag\\
 M_3& = &-\Big(2 \rho  (\tau +1) (2 \alpha +\tau -1) \left(2 \lambda  (\alpha +\tau -1) (\alpha +\tau ) (\alpha +\tau +1)\right.\notag\\
 &&\left.-(\alpha -1) (3 \alpha +3 \tau +1)\right)\Big)\notag\\
 &&/\Big((\alpha +\tau -1) (\alpha +\tau )^3 (2 \alpha +2 \tau -1) (2 \alpha +2 \tau
   +1)\Big)\notag\\
   &&-\frac{2 (\alpha +\tau +1) \left(\lambda  (\alpha +\tau -1) (\alpha +\tau )-\alpha +1\right)^3}{(\alpha +\tau -1)^3 (\alpha +\tau )^3},\notag\\
 M_4& = &-\frac{\rho ^2 (\tau +1)^2 (2 \alpha +\tau -1)^2 (2 \alpha +2 \tau +3)}{(\alpha +\tau )^4 (2
   \alpha +2 \tau -1)^2 (2 \alpha +2 \tau +1)^2} \notag\\
   && -\frac{(2 \alpha +2 \tau +3) \left(\lambda  (\alpha +\tau -1) (\alpha +\tau )-\alpha +1\right)^4}{(\alpha +\tau -1)^4 (\alpha +\tau )^4}+\Big(\rho  (\tau+1) (2 \alpha +\tau -1)\notag\\
 &&  \Big(
  3 \lambda ^2 (\alpha +\tau -1)^2 (\alpha +\tau )^2 (\alpha +\tau +1) (2 \alpha +2 \tau +1) (2 \alpha +2 \tau +3)\notag\\
  &&
  -4 (\alpha -1) \lambda  (\alpha+\tau -1) (\alpha +\tau ) (2 \alpha +2 \tau +1)^2 (2 \alpha +2 \tau +3)\notag\\
  &&
  \left(15 \alpha ^2-46 \alpha +26\right) \tau ^3+\left(65 \alpha ^3-107 \alpha ^2+4 \alpha +36\right) \tau
   ^2\notag\\
   &&+2 (\alpha -1) \left(35 \alpha ^3-\alpha ^2-25 \alpha -11\right) \tau \notag\\
   &&+2 (\alpha -1)^2 \left(12 \alpha
   ^3+18 \alpha ^2+11 \alpha +3\right)+(1-5 \alpha ) \tau ^4-\tau ^5
 \Big)\notag\\
   &&/\Big((\alpha +\tau -1)^2 (\alpha +\tau )^4 (\alpha +\tau +1) (2 \alpha +2 \tau -1) (2 \alpha +2
   \tau +1)^2\Big).\notag \\
E_{n, j}&=& -j-2(1+n+\alpha+\tau), \quad 
F_{n, j}= 2 \lambda -\frac{2 (\alpha -1)}{\alpha +\tau -1}, \notag \\
G_{n,j} & = & \left[ \beta_0^2 + \frac{\gamma_1}{\rho}\bigl(2\tau+2\alpha+2\rho-1\bigr) \right] =\notag\\
&&\Big(2 \left(\lambda ^2 (\alpha+\tau -1)^2 (\alpha +\tau )^2 (2 \alpha +2 \tau +1)\right.\notag\\
   &&-2 (\alpha -1) \lambda  (\alpha+\tau -1) (\alpha +\tau ) (2 \alpha +2 \tau +1)\notag\\
   &&\left. -\tau ^4-2 (2 \alpha -1) \tau ^3-\alpha  (5 \alpha -4) \tau ^2 -2 (\alpha -1) \alpha ^2 \tau
   +2 (\alpha -1)^2\right)\Big)\notag\\
   &&/\Big((\alpha +\tau -1)^2
   (\alpha +\tau )^2 (2 \alpha +2 \tau +1)\Big)\notag\\
&&-\frac{4 \rho  (\tau +1) (2 \alpha +\tau
   -1)}{(\alpha +\tau )^2 (2 \alpha +2 \tau -1) (2 \alpha +2 \tau +1)}.\notag
\end{eqnarray}

\subsubsection{Case 2 analogous to Bessel}
\noindent {\bf Regularity conditions}
$$\lambda,\rho,\alpha \in \mathbb{C},  \quad \rho \neq 0, \quad \alpha \neq \displaystyle\frac{n+3}{2}, \quad n\geq 0.$$

\noindent {\bf Functional equation}
\begin{eqnarray}
&&\Phi(x) = x^2, \quad \psi(x) = -2\alpha x - 2, \notag \\
&&B(x) = (2\alpha-1)x^2 + \{2(1-\alpha)\lambda + 2\}x - 2\lambda - \rho(2\alpha-3). \notag
\end{eqnarray}

\noindent {\bf Moments recurrence coefficients}
\begin{eqnarray}
 M_1& = &2 (\alpha -1) \lambda ,\quad
 M_2 =  (2 \alpha -3) \left(\lambda ^2+\rho \right) ,\notag\\
 M_3& = & 2 (\alpha -2) \lambda ^3+4 (\alpha -2) \lambda  \rho -2 \rho ,\notag\\
 M_4& = & (2 \alpha -5) \lambda
   ^4+3 (2 \alpha -5) \lambda ^2 \rho -\frac{2 (2 \alpha -5) \lambda  \rho }{\alpha -2} +
   \frac{\rho  \left((\alpha -2) (2 \alpha -5) \rho +2\right)}{\alpha -2}.\notag\\
E_{n, j}& = & -4-j-2n+2\alpha, \quad
F_{n, j}=2 (\lambda +1), \quad
G_{n, j}=2 \left(\lambda ^2+2 \rho \right).\notag
\end{eqnarray}

\subsubsection{Case 3 analogous to Bessel}
\noindent {\bf Regularity conditions}
$$\lambda, \rho, \alpha \in\mathbb{C}, \quad \rho\neq 0, \quad \alpha\neq \displaystyle\frac{1-n}{2}, \quad n\geq 0.$$

\noindent {\bf Functional equation}
\begin{eqnarray}
&&\Phi(x)=x^2, \quad
\psi(x)=2(\alpha-2)x+2, \notag \\
&&B(x)=-(2\alpha-3)x^2+2\{(\alpha-1)\lambda-1\}x+2\lambda+\rho(2\alpha-1). \notag
\end{eqnarray}

\noindent {\bf Moments recurrence coefficients}
\begin{eqnarray}
 M_1& = & -2 (\alpha -1) \lambda,\quad
 M_2 =  -(2 \alpha -1) \left(\lambda ^2+\rho \right) ,\quad
 M_3 =  2 \rho -2 \alpha  \left(\lambda ^3+2 \lambda  \rho \right) ,\notag\\
 M_4& = & (-2 \alpha -1) \lambda ^4-3 (2 \alpha
   +1) \lambda ^2 \rho +\frac{2 (2 \alpha +1) \lambda  \rho }{\alpha }
    -\frac{\rho  \left(\alpha  (2 \alpha +1) \rho +2\right)}{\alpha }.\notag\\
E_{n, j} & =& -j-2(n+\alpha), \quad
F_{n, j}=2 (\lambda -1), \quad
G_{n, j}=2 \left(\lambda ^2+2 \rho \right).\notag
\end{eqnarray}

\subsubsection{Case 4 analogous to Bessel}
\noindent {\bf Regularity conditions}
$$\lambda, \rho, \alpha \in\mathbb{C}, \quad \rho\neq0, \quad \alpha\neq \displaystyle\frac{1-n}{2}, \quad n\geq-1.$$

\noindent {\bf Functional equation}
\begin{eqnarray}
&&\Phi(x)=x^2,\quad
\psi(x)=2\displaystyle\left\{\frac{1-\rho}{\rho}+\alpha\frac{\rho-2}{\rho}\right\} x+\frac{2}{\rho}(2\alpha-1)\beta_0-\frac{2}{\alpha}, \notag \\
&&B(x)=\displaystyle\frac{\rho-1}{\rho}(1-2\alpha)x^2+2\left\{\left[\alpha+\frac{1}{\rho}(1-2\alpha)\right]\beta_0+\frac{1}{\alpha}\right\} x \notag \\
&&\quad\quad\quad~~
+ \displaystyle\frac{2\alpha-1}{\rho}\beta_0^2 - \frac{2\beta_0}{\alpha} + \frac{\rho}{(2\alpha-1)\alpha^2}. \notag 
\end{eqnarray}

\noindent {\bf Moments recurrence coefficients}
\begin{eqnarray}
 M_1& = & -\frac{2 \left(\alpha ^2 \lambda -\alpha  \lambda +1\right)}{\alpha -1},\notag\\
 M_2& = & -\frac{\rho }{\alpha ^2 (2 \alpha -1)} -\frac{(2 \alpha +1) \left((\alpha -1) \alpha  \lambda +1\right)^2}{(\alpha -1)^2 \alpha
   ^2},\notag\\
 M_3& = & -\frac{2 \rho  \left(2 (\alpha -1) (\alpha +1) \alpha  \lambda +3 \alpha +1\right)}{(\alpha -1) \alpha ^3 (2
   \alpha -1) (2 \alpha +1)}-\frac{2 (\alpha +1) \left((\alpha -1) \alpha  \lambda +1\right)^3}{(\alpha -1)^3 \alpha ^3} ,\notag\\
 M_4& = & -\frac{(2 \alpha +3) \rho ^2}{\alpha ^4 (2 \alpha -1)^2 (2 \alpha +1)^2}
-\frac{(2 \alpha +3) \left(\alpha^2 \lambda -\alpha  \lambda +1\right)^4}{(\alpha -1)^4 \alpha ^4}\notag\\
   &&-\Big(\rho  \left(3 (\alpha -1)^2 (\alpha +1) (2 \alpha +1) (2 \alpha +3) \alpha ^2 \lambda ^2
  \right.\notag\\
   &&\left.+4(\alpha -1) (2 \alpha +1)^2 (2 \alpha +3) \alpha  \lambda +25 \alpha ^3+37 \alpha ^2+22 \alpha +6\right)\Big)\notag\\
&&/\Big((\alpha -1)^2 \alpha ^4 (\alpha +1) (2 \alpha -1) (2 \alpha +1)^2\Big) .\notag\\
E_{n, j}&=& -j-2(n+\alpha+1), \quad
F_{n, j}= 2 \left(\frac{1}{\alpha -1}+\lambda \right), \notag\\
G_{n, j}&=& 2\left[ \beta_0^2 + \frac{2\alpha+2\rho-1}{\alpha^2(4\alpha^2-1)} \right]=\notag\\
&&\frac{4 \rho }{\alpha ^2 (2 \alpha -1) (2 \alpha +1)}\notag\\
&&+\frac{2 \left((\alpha -1)^2 (2 \alpha +1) \alpha ^2 \lambda ^2+\alpha ^2+2 (\alpha -1) (2 \alpha +1) \alpha  \lambda +2\right)}{(\alpha -1)^2 \alpha ^2 (2 \alpha +1)}. \notag
\end{eqnarray}

\subsection{Two cases analogous to the Jacobi case}

\subsubsection{Case 1 analogous to Jacobi}
\noindent {\bf Regularity conditions}
\begin{eqnarray}
&&\lambda, \rho, \alpha, \beta, \tau \in\mathbb{C},\quad \rho\neq0, \notag \\
&&\tau\neq-n-1,\quad \tau+\alpha\neq-n-1,\quad \tau+\beta\neq-n-1,\quad 2\tau+\alpha+\beta\neq-n,\quad n\geq 0. \notag
\end{eqnarray}

\noindent {\bf Functional equation}
\begin{eqnarray}
&&\Phi(x)=x^2-1, \notag \\
&&\psi(x)=\left\{\frac{\rho-2}{\rho}(2\tau+\alpha+\beta+1)-1\right\}x+\frac{2}{\rho}(2\tau+\alpha+\beta+1)\beta_0-\frac{\alpha^2-\beta^2}{2\tau+\alpha+\beta+2}, \notag \\
&&B(x)=\frac{1-\rho}{\rho}(2\tau+\alpha+\beta+1)x^2+\left\{\left[\frac{\rho-2}{\rho}(2\tau+\alpha+\beta)+2\frac{\rho-1}{\rho}\right]\beta_0+\frac{\alpha^2-\beta^2}{2\tau+\alpha+\beta+2}\right\}x \notag \\
&&\quad\quad\quad\quad+(2\tau+\alpha+\beta+3)\gamma_1-1+\frac{1}{\rho}(2\tau+\alpha+\beta+1)\beta_0^2-\frac{\alpha^2-\beta^2}{2\tau+\alpha+\beta+2}\beta_0. \notag 
\end{eqnarray}

\noindent {\bf Moments recurrence coefficients}
\begin{eqnarray}
 M_1& = &\lambda  (-\alpha -\beta -2 \tau -2)-\frac{(\alpha -\beta ) (\alpha +\beta )}{\alpha +\beta +2 \tau } ,\notag\\
 M_2& = &-\frac{4 \rho  (\tau +1) (\alpha +\tau +1) (\beta +\tau +1) (\alpha +\beta +\tau
   +1)}{(\alpha +\beta +2 \tau +1) (\alpha +\beta +2 \tau +2)^2} \notag\\
   &&-\frac{(\alpha +\beta +2 \tau +3) \left(\lambda  (\alpha +\beta +2 \tau ) (\alpha +\beta +2 \tau +2)+(\alpha -\beta ) (\alpha +\beta ) \right)^2}{(\alpha
   +\beta +2 \tau )^2 (\alpha +\beta +2 \tau +2)^2} ,\notag\\
 M_3& = & -\frac{(\alpha +\beta +2 \tau +4) \left(\lambda  (\alpha +\beta +2 \tau ) (\alpha +\beta +2 \tau +2)+(\alpha -\beta ) (\alpha +\beta )\right)^3}{(\alpha
   +\beta +2 \tau )^3 (\alpha +\beta +2 \tau +2)^3}\notag\\
&&-\Big(4 \rho  (\tau +1) (\alpha +\tau +1) (\beta +\tau +1) (\alpha +\beta +\tau +1)\notag\\
&&
   \left(2 \lambda  (\alpha +\beta +2 \tau ) (\alpha +\beta +2 \tau +2) (\alpha +\beta +2 \tau +4)\right.\notag\\
   &&\left.+(\alpha -\beta ) (\alpha +\beta ) (3 \alpha +3
   \beta +6 \tau +8)\right)\Big)\notag\\
   &&/\Big((\alpha +\beta +2 \tau ) (\alpha +\beta +2 \tau +1) (\alpha +\beta
   +2 \tau +2)^3 (\alpha +\beta +2 \tau +3)\Big) ,\notag\\
 M_4& = &  -\frac{(\alpha +\beta +2 \tau +5) \left(\lambda  (\alpha +\beta +2 \tau ) (\alpha +\beta +2 \tau +2)+(\alpha -\beta ) (\alpha +\beta ) \right)^4}{(\alpha
   +\beta +2 \tau )^4 (\alpha +\beta +2 \tau +2)^4}\notag\\
   &&-\frac{16 \rho ^2 (\tau +1)^2 (\alpha +\tau +1)^2 (\beta +\tau +1)^2 (\alpha +\beta
   +\tau +1)^2 (\alpha +\beta +2 \tau +5)}{(\alpha +\beta +2 \tau +1)^2 (\alpha +\beta +2 \tau +2)^4 (\alpha +\beta +2 \tau
   +3)^2}\notag\\
   &&-\Big(4 \rho  (\tau +1) (\alpha +\tau +1) (\beta +\tau +1) (\alpha +\beta +\tau +1) \notag\\
   &&\left(3 \lambda ^2 (\alpha +\beta +2 \tau )^2 (\alpha +\beta +2 \tau +3) (\alpha +\beta +2 \tau +4) (\alpha +\beta +2 \tau +5) (\alpha +\beta
   +2 \tau +2)^2\right.\notag\\
     &&+8 (\alpha -\beta ) (\alpha +\beta ) \lambda  (\alpha +\beta +2 \tau ) (\alpha +\beta +2 \tau +3)^2 (\alpha +\beta +2\tau +5) (\alpha +\beta +2 \tau +2)\notag\\
&&  +2 \left(16 \tau ^7+8 \tau ^6 (7 \alpha +7 \beta +16)+4 \tau ^5 \left(19 \alpha ^2+42 \alpha  \beta +96 \alpha +19 \beta ^2+96 \beta +100\right)\right.\notag\\
&&+2 \tau ^4 \left(25 \alpha ^3+95 \alpha ^2
   \beta +224 \alpha ^2+95 \alpha  \beta ^2+480 \alpha  \beta +500 \alpha +25 \beta ^3+224 \beta ^2+500 \beta +304\right)\notag\\
&&+4 \tau ^3
   \left(10 \alpha ^4+25 \alpha ^3 \beta +64 \alpha ^3+30 \alpha ^2 \beta ^2+224 \alpha ^2 \beta +240 \alpha ^2+25 \alpha  \beta ^3+224
   \alpha  \beta ^2\right.\notag\\
   &&\left.+500 \alpha  \beta +304 \alpha +10 \beta ^4+64 \beta ^3+240 \beta ^2+304 \beta +112\right)
   \notag\\
   &&+2 \tau ^2 \left(19 \alpha
   ^5+30 \alpha ^4 \beta +92 \alpha ^4-5 \alpha ^3 \beta ^2+192 \alpha ^3 \beta +220 \alpha ^3-5 \alpha ^2 \beta ^3+200 \alpha ^2 \beta
   ^2 \right.\notag\\
   &&+720 \alpha ^2 \beta+448 \alpha ^2+30 \alpha  \beta ^4+192 \alpha  \beta ^3+720 \alpha  \beta ^2+912 \alpha  \beta +336 \alpha +19
   \beta ^5+92 \beta ^4\notag\\
   &&\left.+220 \beta ^3+448 \beta ^2+336 \beta +64\right)\notag\\
   &&+2 \tau  (\alpha +\beta ) \left(9 \alpha ^5+10 \alpha ^4 \beta +60
   \alpha ^4-15 \alpha ^3 \beta ^2+32 \alpha ^3 \beta +133 \alpha ^3-15 \alpha ^2 \beta ^3\right.\notag\\
   &&-56 \alpha ^2 \beta ^2+87 \alpha ^2 \beta +144
   \alpha ^2+10 \alpha  \beta ^4+32 \alpha  \beta ^3+87 \alpha  \beta ^2+304 \alpha  \beta +168 \alpha \notag\\
   &&\left.+9 \beta ^5+60 \beta ^4+133 \beta^3+144 \beta ^2+168 \beta +64\right)\notag\\
   &&+(\alpha +\beta )^2 \left(3 \alpha ^5+3 \alpha ^4 \beta +28 \alpha ^4-6 \alpha ^3 \beta ^2+4
   \alpha ^3 \beta +93 \alpha ^3-6 \alpha ^2 \beta ^3-48 \alpha ^2 \beta ^2\right.\notag\\
   &&-53 \alpha ^2 \beta +124 \alpha ^2+3 \alpha  \beta ^4+4
   \alpha  \beta ^3-53 \alpha  \beta ^2-104 \alpha  \beta +56 \alpha +3 \beta ^5+28 \beta ^4\notag\\
   &&\left.\left. \left.+93 \beta ^3+124 \beta ^2+56 \beta+32\right)\right)\right)\Big)\notag\\
   &&/\Big((\alpha +\beta +2 \tau )^2 (\alpha +\beta +2 \tau +1) (\alpha +\beta +2 \tau +2)^4 (\alpha +\beta +2 \tau +3)^2 (\alpha+\beta +2 \tau +4)\Big).\notag
\end{eqnarray}
\begin{align*}
E_{n, j}=&-\bigl(2n + 2\tau + \alpha + \beta + 4+  j \bigr), \quad
F_{n, j}=\frac{(\alpha -\beta ) (\alpha +\beta )}{\alpha +\beta +2 \tau }+2 \lambda,\\
G_{n, j}=& 2n + j + 2\beta_0^2 + 2\gamma_1 \left[ \frac{1}{\rho}(2\tau+\alpha+\beta+1) + 2 \right] \\
=& j +\frac{16 \rho  (\tau +1) (\alpha +\tau +1) (\beta +\tau +1) (\alpha +\beta +\tau +1)}{(\alpha +\beta +2 \tau +1) (\alpha +\beta +2 \tau+2)^2 (\alpha +\beta +2 \tau +3)}\notag\\
   &+\Big(2 \left(16 \tau ^6+16 \tau ^5 (3 \alpha +3 \beta +2 n+4)\right.\notag\\
   &+4 \tau ^4 \left(13 \alpha ^2+30 \alpha  \beta +40 \alpha +13 \beta ^2+40 \beta +4 n (5 \alpha +5 \beta +7)+24\right)\notag\\
   &+8 \tau ^3 \left(3 \alpha ^3+13 \alpha ^2 \beta +18 \alpha ^2+13 \alpha  \beta ^2+40 \alpha  \beta +24 \alpha +3 \beta ^3+18 \beta ^2+24 \beta \right.\notag\\
   &\left.+2 n(\alpha +\beta +2) (5 \alpha +5 \beta +4)+8\right)\notag\\
  &+4 \tau ^2 \left(\alpha ^4+9 \alpha ^3 \beta +14 \alpha ^3+16 \alpha ^2 \beta ^2+54 \alpha ^2 \beta +34 \alpha ^2+9 \alpha  \beta ^3+54 \alpha  \beta^2+72 \alpha  \beta \right.\notag\\
   &+24 \alpha +\beta ^4+14 \beta ^3+34 \beta ^2+24 \beta \notag\\
   &\left.+2 n \left(5 \alpha ^3+15 \alpha ^2 \beta +21 \alpha ^2+15\alpha  \beta ^2+42 \alpha  \beta +24 \alpha +5 \beta ^3+21 \beta ^2+24 \beta +6\right)+4\right)\notag\\
&+2 \tau  (\alpha+\beta ) (\alpha +\beta +2) \left(2 \alpha ^2 \beta +5 \alpha ^2+2 \alpha  \beta ^2+14 \alpha  \beta +10 \alpha +5 \beta ^2+10 \beta \right.\notag\\
&\left.+n \left(5 \alpha ^2+10 \alpha  \beta+18 \alpha +5 \beta ^2+18 \beta +12\right)+4\right)\notag\\
  &+(\alpha +\beta )^2 \left(\alpha ^3+3 \alpha ^2 \beta +7 \alpha ^2+3 \alpha  \beta ^2+6 \alpha  \beta +8 \alpha +\beta ^3+7 \beta ^2+8 \beta \right.\notag\\
   &\left.+n (\alpha +\beta+2)^2 (\alpha +\beta +3)+4\right)\notag\\
    &+\lambda ^2 (\alpha +\beta +2 \tau )^2 (\alpha +\beta +2 \tau +2)^2 (\alpha +\beta +2 \tau +3)\notag\\
    &\left.+2 (\alpha -\beta ) (\alpha +\beta ) \lambda  (\alpha +\beta +2 \tau ) (\alpha +\beta +2 \tau +2) (\alpha +\beta +2 \tau +3) \right)\Big)\notag\\
   &/\Big((\alpha +\beta +2 \tau )^2 (\alpha +\beta +2 \tau +2)^2
   (\alpha +\beta +2 \tau +3)\Big). 
\end{align*}

\subsubsection{Case 2 analogous to Jacobi}
\noindent {\bf Regularity conditions}
$$\lambda, \rho, \alpha, \beta \in \mathbb{C}, \quad \rho \neq 0, \quad \alpha \neq -n, \quad \beta \neq -n, \quad \alpha + \beta \neq -n, \quad n \geq 1.$$

\noindent {\bf Functional equation}
\begin{eqnarray}
&&\Phi(x) = x^2 - 1, \quad
\psi(x) = (\alpha + \beta - 2)x + \alpha - \beta,  \notag \\
&&B(x) = (1 - \alpha - \beta)x^2 + ((\alpha + \beta)\lambda + \beta - \alpha)x + (\alpha - \beta)\lambda - 1 + \rho(1 + \alpha + \beta). \notag
\end{eqnarray}

\noindent {\bf Moments recurrence coefficients}
\begin{eqnarray}
 M_1& = & -\lambda  (\alpha +\beta ),\quad
 M_2 = -(\alpha +\beta +1) \left(\lambda ^2+\rho \right) ,\notag\\
 M_3& = & \lambda ^3 (-\alpha -\beta -2)-2 \lambda  \rho  (\alpha +\beta +2)+\rho  (\alpha -\beta ) ,\notag\\
 M_4& = & \lambda ^4 (-\alpha
   -\beta -3)-3 \lambda ^2 \rho  (\alpha +\beta +3)+\frac{2 \lambda  \rho  (\alpha -\beta ) (\alpha +\beta
   +3)}{\alpha +\beta +2}\notag\\
 && -\frac{\rho  \left(\rho  (\alpha +\beta +2) (\alpha +\beta +3)+\alpha ^2-2 \alpha  \beta +\alpha +\beta ^2+\beta +2\right)}{\alpha +\beta +2}.\notag\\
E_{n, j}&=&-(2n + \alpha + \beta + 2 + j), \quad
F_{n, j}=-\alpha +\beta +2 \lambda, \notag\\
G_{n, j}& = & j+ 2(n + \lambda^2 + 2\rho).\notag
\end{eqnarray}

\section{Moments of Laguerre-Hahn forms of class~0}\label{Section5}

The formulas for the first four moments and the general recurrence relation established in Theorem \ref{thm:main} of the previous section, together with the coefficients corresponding to each family, can be implemented in a computer algebra system to generate the moments recursively.

In this section, we first present an algorithm for computing the moments $(u)_n$ of Laguerre–Hahn forms of class zero. We then provide several moments for each family, obtained through an implementation of this algorithm in {\it Mathematica$^{\circledR}$}. As the examples illustrate, the complexity of the resulting expressions grows rapidly with 
$n$. Consequently, the symbolic simplification and factorization capabilities of {\it Mathematica$^{\circledR}$} are indispensable for expressing these results in a manageable and comprehensible form.

\subsection{Algorithm}

\vspace{0.25cm}

\noindent {\bf Algorithm for the recursive computation of the moments of Laguerre-Hahn forms of class~0}

\vspace{0.25cm}

\begin{enumerate} 

\item {\bf Input Data}

\vspace{0.25cm}

\begin{enumerate}

\item Specify the parameters defining the form $u$: $\tau$, $\lambda$, $\rho$, $\alpha$, or $\beta$.

\vspace{0.25cm}

\item State the first moment $(u)_0=1$.

\vspace{0.25cm}

\item State the maximum number of moments to be computed, denoted by $2 nmax+1$, with $nmax\geq 2$.

\end{enumerate}

\vspace{0.25cm}

\item {\bf Computations }

\vspace{0.25cm}

\begin{enumerate}

\item Compute the four subsequent moments, $(u)_k$, $k=1,2,3,4$, using equations \eqref{mom-k} with the coefficients $M_k$ and $E_{0,k-3}$ specified for each family.

\vspace{0.25cm}

\item For $n=1,\ldots,nmax-1$, 

\vspace{0.25cm}

\begin{enumerate}
\item
Compute  $(u)_{2n+3}$ using \eqref{eq:decouple-j}, for $j=0$.

\vspace{0.25cm}

\item
Compute  $(u)_{2n+4}$ using \eqref{eq:decouple-j}, for $j=1$.

\end{enumerate}

\end{enumerate}

\vspace{0.25cm}

\item {\bf Results}

\vspace{0.25cm}

$(u)_n$, $n=0,...,2 nmax$. Moments should be written as polynomials in the parameters, with coefficients factorized whenever possible.

\end{enumerate}

\subsection{Moments}

\noindent {\bf Moments of Case~1 analogous to Hermite}
\begin{eqnarray}
(u)_0 & = & 1,\quad (u)_1= \lambda,\quad (u)_2 =\lambda ^2+\frac{1}{2} \rho  (\tau +1),\quad
(u)_3= \lambda ^3+\lambda  \rho  (\tau +1),\notag\\
(u)_4& =& \lambda ^4+\frac{3}{2} \lambda ^2 \rho  (\tau +1)+\frac{1}{4} \rho  (\tau +1) (\rho  (\tau +1)+\tau +2),\notag\\
(u)_5 & =& \lambda ^5+2 \lambda ^3 \rho  (\tau +1)+\frac{1}{4} \lambda  \rho  (\tau +1) (3 \rho  (\tau +1)+2 (\tau +2)),\notag\\
(u)_6 & =&\lambda ^6+\frac{5}{2} \lambda ^4 \rho  (\tau +1)+\frac{3}{4} \lambda ^2 \rho  (\tau +1)
   (2 \rho  (\tau +1)+\tau +2)\notag\\
   &&+\frac{1}{8} \rho  (\tau +1) \left(\rho ^2 (\tau +1)^2+2 \rho  (\tau +2) (\tau +1)+(\tau +2) (2 \tau +5)\right),\notag\\
(u)_7 & =& \lambda ^7+3 \lambda ^5 \rho  (\tau +1)+\frac{1}{2} \lambda ^3 \rho  (\tau +1) (5 \rho  (\tau +1)+2 (\tau +2))\notag\\
   &&+\frac{1}{4} \lambda  \rho  (\tau +1) \left(2 \rho ^2 (\tau +1)^2+3 \rho  (\tau +2) (\tau +1)+(\tau +2) (2 \tau +5)\right)  ,\notag\\
(u)_8 & =& \lambda ^8+\frac{7}{2} \lambda ^6 \rho  (\tau +1)+\frac{5}{4} \lambda ^4 \rho  (\tau +1)
   (3 \rho  (\tau +1)+\tau +2)\notag\\
   &&+\frac{1}{8} \lambda ^2 \rho  (\tau +1) \left(10 \rho ^2 (\tau +1)^2+12 \rho  (\tau +2) (\tau +1)+3 (\tau +2) (2 \tau +5)\right)\notag\\
   &&+\frac{1}{16} \rho  (\tau +1) \left(\rho ^3 (\tau +1)^3+3 \rho ^2 (\tau +2) (\tau +1)^2\right.\notag\\
   &&\left.+\rho  (\tau +2) (5 \tau +12) (\tau
   +1)+(\tau +2) \left(5 \tau ^2+27 \tau +37\right)\right)  .\notag
\end{eqnarray}

\noindent {\bf Moments of Case~2 analogous to Hermite}
\begin{eqnarray}
(u)_0 & = & 1,\quad (u)_1= \lambda,\quad (u)_2=\lambda ^2+\frac{\rho }{2},\quad
(u)_3 =\lambda ^3+\lambda  \rho,\quad
(u)_4=\lambda ^4+\frac{3 \lambda ^2 \rho }{2}+\frac{1}{4} \rho  (\rho +1),\notag\\
(u)_5  &=  &\lambda ^5+2 \lambda ^3 \rho +\frac{1}{4} \lambda  \rho  (3 \rho +2) ,\quad
(u)_6  =\lambda ^6+\frac{5 \lambda ^4 \rho }{2}+\frac{3}{4} \lambda ^2 \rho  (2 \rho
   +1)+\frac{1}{8} \rho  \left(\rho ^2+2 \rho +3\right),\notag\\
(u)_7  &=  & \lambda ^7+3 \lambda ^5 \rho +\frac{1}{2} \lambda ^3 \rho  (5 \rho +2)+\frac{1}{4}
   \lambda  \rho  \left(2 \rho ^2+3 \rho +3\right),\notag\\
(u)_8 &=  & \lambda ^8+\frac{7 \lambda ^6 \rho }{2}+\frac{5}{4} \lambda ^4 \rho  (3 \rho
   +1)+\frac{1}{8} \lambda ^2 \rho  \left(10 \rho ^2+12 \rho +9\right)+\frac{1}{16} \rho
    \left(\rho ^3+3 \rho ^2+7 \rho +15\right).\notag
\end{eqnarray}

\noindent {\bf Moments of Case~1 analogous to Laguerre}
\begin{eqnarray}
(u)_0 & = & 1,\quad (u)_1=\alpha +\lambda +2 \tau +1,\quad 
(u)_2= \rho  (\tau +1) (\alpha +\tau +1)+(\alpha +\lambda +2 \tau +1)^2, \notag\\
(u)_3& =& \rho  (\tau +1) (\alpha +\tau +1) (3 \alpha +2 \lambda +6 \tau +5)+(\alpha +\lambda +2
   \tau +1)^3\notag\\
(u)_4& =& \rho ^2 (\tau +1)^2 (\alpha +\tau +1)^2+\rho  (\tau +1) (\alpha +\tau +1) \notag\\
&&\left(4 \lambda  (2 \alpha +4 \tau +3)+6 \alpha ^2+(25 \alpha +44) \tau +22 \alpha +3 \lambda ^2+25 \tau ^2+22\right)\notag\\
   &&+(\alpha +\lambda +2 \tau +1)^4, \notag\\
(u)_5& =&\rho ^2 (\tau +1)^2 (\alpha +\tau +1)^2 (5 \alpha +3 \lambda
   +10 \tau +9)+\rho  (\tau +1) (\alpha +\tau +1)\notag\\
   && \left(3 \lambda ^2 (5 \alpha +10 \tau
   +7)+2 \lambda  \left(2 \left(5 \alpha ^2+16 \alpha +14\right)
   +(41 \alpha +64) \tau +41 \tau ^2\right)\right.\notag\\
   &&\left(65 \alpha ^2+253 \alpha +272\right) \tau +2 \left(5 \alpha ^3+30 \alpha ^2+68
   \alpha +2 \lambda ^3+55\right)\notag\\
   &&\left.+(135 \alpha +253) \tau ^2+90 \tau ^3\right)+(\alpha +\lambda +2 \tau +1)^5, \notag\\
(u)_6& =&\rho ^3 (\tau +1)^3 (\alpha +\tau +1)^3+\rho ^2 (\tau +1)^2 (\alpha +\tau +1)^2 \notag\\
&&\left(15 \alpha ^2+6 \lambda  (3 \alpha +6 \tau +5)+62 \alpha  \tau +58 \alpha +6 \lambda
   ^2+178 \tau +9\right)\notag\\
   &&+\rho  (\tau +1) (\alpha +\tau +1) \left(8 \lambda ^3 (3 \alpha +6 \tau +4)\right.\notag\\
   &&+3 \lambda ^2 \left(15 \alpha ^2+(61 \alpha +88) \tau +44 \alpha +61 \tau ^2+35\right)\notag\\
   &&+4 \lambda  \left(\left(63 \alpha ^2+211 \alpha +196\right) \tau +(129 \alpha
   +211) \tau ^2 \right.\notag\\
   &&\left.+86 \tau ^3 +10 \alpha ^3+51 \alpha ^2+98 \alpha +69\right)\notag\\
   &&+302 \tau ^4+ 2 (302 \alpha +607) \tau ^3+\left(437 \alpha ^2+1821 \alpha +2112\right) \tau ^2\notag\\
   &&+\left(135 \alpha ^3+867 \alpha ^2+2112 \alpha +1840\right) \tau \notag\\
   &&\left.+496\alpha ^2+
   +15 \alpha ^4+130 \alpha ^3+920 \alpha +5 \lambda ^4+6599\right)\notag\\
   &&+(\alpha +\lambda +2 \tau+1)^6,\notag\\
   (u)_7& =& \rho ^3 (\tau +1)^3 (\alpha +\tau +1)^3 (7 \alpha +4
   \lambda +14 \tau +13)+\rho ^2 (\tau +1)^2 (\alpha +\tau +1)^2\notag\\
   && \left(10 \lambda ^3+6 \lambda ^2 (7 \alpha +14 \tau +11)+3 \lambda  \left(21 \alpha ^2+86 \alpha  \tau +74 \alpha +86 \tau ^2+148\tau +69\right)\right.\notag\\
   &&\left.+308 \tau ^3+22 (21\alpha +41) \tau ^2+2 \left(112 \alpha ^2+451 \alpha +481\right) \tau \right.\notag\\
   &&\left.+35 \alpha ^3+217 \alpha ^2+481 \alpha +1371\right)\notag\\
   &&+\rho  (\tau +1) (\alpha +\tau +1)
   \left(5 \lambda ^4 (7 \alpha +14 \tau +9)\right.\notag\\
   &&\left.+4 \lambda ^3 \left(85 \tau ^2+(85 \alpha +116) \tau +21 \alpha ^2+58 \alpha +43\right)\right.\notag\\
   &&+3 \lambda ^2 \left(294 \tau ^3+(441 \alpha
   +659) \tau ^2+\left(217 \alpha ^2+659 \alpha +550\right) \tau \right.\notag\\
   &&\left.+35 \alpha ^3+161 \alpha ^2+275 \alpha +173\right)\notag\\
   &&+2 \lambda  \left(2 (646\alpha +1115) \tau ^3+\left(947 \alpha ^2+3345 \alpha +3318\right) \tau ^2\right.\notag\\
   &&+\left(301 \alpha ^3+1619 \alpha ^2+3318 \alpha +2508\right) \tau \notag\\
   &&\left.+ 35 \alpha ^4+252 \alpha ^3+794
   \alpha ^2+1254 \alpha +646 \tau ^4+797\right)+(\alpha +\lambda +2 \tau +1)^7.\notag
\end{eqnarray}

\noindent {\bf Moments of Case~2 analogous to Laguerre}
\begin{eqnarray}
(u)_0&=& 1,\quad (u)_1=\alpha +\lambda -1,\quad (u)_2=\rho+(\alpha +\lambda -1)^2,\notag\\
(u)_3& =&  \rho  (3 \alpha +2 \lambda -1)+(\alpha +\lambda -1)^3,\notag\\
(u)_4& =& \rho ^2+\rho  \left(3 \left(2 \alpha ^2-\alpha +1\right)+4 (2 \alpha -1) \lambda +3 \lambda ^2\right)+(\alpha +\lambda -1)^4 ,\notag\\
(u)_5& =&\rho ^2 (5 \alpha +3 \lambda -1)
+\rho  \left(4 \lambda ^3+3 (5 \alpha -3) \lambda ^2+2 \left(10 \alpha ^2-9 \alpha +5\right) \lambda\right.\notag\\
&&\left.+10 \alpha ^3 -5 \alpha ^2+18 \alpha +1\right)+(\alpha +\lambda -1)^5,\notag\\
(u)_6& =& \rho ^3+\rho ^2 \left(15 \alpha ^2+6 (3 \alpha -1) \lambda -4 \alpha +6 \lambda ^2+5\right)\notag\\
   &&+\rho  \left(5 \lambda ^4+8 (3 \alpha -2)\lambda ^3+3 \left(15 \alpha ^2-17 \alpha +8\right) \lambda ^2+8 \left(5 \alpha ^3-6 \alpha ^2+8 \alpha -1\right) \lambda\right.\notag\\
   &&\left. +66 \alpha
   ^2+15 \alpha ^4-5 \alpha ^3+25 \alpha +19\right)+(\alpha +\lambda -1)^6,\notag\\
(u)_7& =& \rho ^3 (7 \alpha +4 \lambda -1)
+\rho ^2 \left(10 \lambda ^3+6 (7 \alpha-3) \lambda ^2+3 \left(21 \alpha ^2-12 \alpha +7\right) \lambda \right.\notag\\
&&\left.+35 \alpha ^3-7 \alpha ^2+41 \alpha +3\right)+\rho 
   \left(6 \lambda ^5+5 (7 \alpha -5) \lambda ^4+12 \left(7 \alpha ^2-9 \alpha +4\right) \lambda ^3\right.\notag\\
   &&+3 \left(35 \alpha ^3-56 \alpha ^2+57\alpha -12\right) \lambda ^2+2 \left(35 \alpha ^4-49 \alpha ^3+122 \alpha ^2-11 \alpha +23\right)
   \lambda \notag\\
   &&\left.+3 \left(7 \alpha ^5+63 \alpha ^3+61 \alpha ^2+82 \alpha
   +27\right)\right)+(\alpha
   +\lambda -1)^7, \notag\\
(u)_8& =& \rho ^4+\rho ^3 \left(28 \alpha ^2+8 (4 \alpha -1) \lambda -5 \alpha +10 \lambda ^2+7\right)
+\rho ^2 \left(15 \lambda ^4+40 (2 \alpha -1) \lambda
   ^3\right.\notag\\
&&\left.+12 \left(14 \alpha ^2-11 \alpha +5\right) \lambda ^2+12
   \left(14 \alpha ^3-10 \alpha ^2+15 \alpha -1\right) \lambda \right.\notag\\
   &&\left.+70 \alpha ^4+195 \alpha ^2+66 \alpha +41\right)+\rho  \left(7 \lambda ^6+12 (4 \alpha -3) \lambda ^5 
   +5 \left(28 \alpha ^2-39 \alpha
   +17\right) \lambda ^4\right.\notag\\
   &&+32 \left(7 \alpha ^3-13 \alpha ^2+12 \alpha
   -3\right) \lambda ^3+15 \left(14 \alpha ^4-28 \alpha ^3+47 \alpha ^2-16 \alpha+7\right) \lambda ^2 \notag\\
   && \left. +4 \left(28 \alpha ^5-42 \alpha ^4+180 \alpha ^3+25 \alpha ^2+140
   \alpha +29\right) \lambda \right.\notag\\
   &&\left.+28 \alpha ^6+14 \alpha ^5+462 \alpha ^4+835 \alpha ^3+1715 \alpha ^2+1447 \alpha +539\right)+(\alpha
   +\lambda -1)^8 . \notag
\end{eqnarray}

\noindent {\bf Moments of Case~1 analogous to Bessel}
\begin{eqnarray}
(u)_0 & = & 1, \quad (u)_1= \frac{\lambda  (\alpha +\tau -1) (\alpha +\tau )-\alpha +1}{(\alpha +\tau -1) (\alpha +\tau )},\notag\\
(u)_2 &= & 
-\frac{\rho 
   (\tau +1) (2 \alpha +\tau -1)}{(\alpha +\tau )^2 (2 \alpha +2 \tau -1) (2 \alpha +2
   \tau +1)}+\frac{\left(\lambda  (\alpha +\tau -1) (\alpha +\tau )-\alpha +1\right)^2}{(\alpha +\tau -1)^2 (\alpha +\tau )^2} ,\notag\\
(u)_3 &= & 
-\Big(\rho
    (\tau +1) (2 \alpha +\tau -1) \left(2 \lambda  (\alpha +\tau -1) (\alpha +\tau ) (\alpha +\tau +1)\right.\notag\\
    &&\left.-(\alpha -1) (3 \alpha +3
   \tau +1)\right)\Big)\notag\\
   &&/\Big((\alpha +\tau -1) (\alpha
   +\tau )^3 (\alpha +\tau +1) (2 \alpha +2 \tau -1) (2 \alpha +2 \tau +1)\Big)\notag\\
&&   +\frac{\left(\lambda  (\alpha +\tau -1) (\alpha +\tau )-\alpha +1\right)^3}{(\alpha +\tau -1)^3 (\alpha +\tau )^3},\notag\\
(u)_4 & = & \frac{\rho^2 (\tau +1)^2 (2 \alpha +\tau -1)^2}{(\alpha +\tau )^4 (2 \alpha +2 \tau -1)^2 (2
   \alpha +2 \tau +1)^2}-\rho(\tau +1) (2 \alpha +\tau -1)\notag\\
    &&\left(3 \lambda ^2 (\alpha +\tau -1)^2 (\alpha +\tau )^2 (\alpha +\tau +1) (2 \alpha +2 \tau
   +1) (2 \alpha +2 \tau +3)\right.\notag\\
   && -4 \lambda(\alpha -1)   (\alpha +\tau -1) (\alpha +\tau ) (2
   \alpha +2 \tau +1)^2 (2 \alpha +2 \tau +3)\notag\\
&&  +\left(15 \alpha ^2-46 \alpha +26\right) \tau ^3+\left(65 \alpha ^3-107 \alpha ^2+4
   \alpha +36\right) \tau ^2\notag\\
   &&+2 (\alpha -1) \left(35 \alpha ^3-\alpha ^2-25 \alpha
   -11\right) \tau \notag\\
   &&\left.+2 (\alpha -1)^2 \left(12 \alpha ^3+18 \alpha ^2+11 \alpha
   +3\right)+(1-5 \alpha ) \tau ^4-\tau ^5\right)\notag\\
   &&/\left((\alpha +\tau -1)^2 (\alpha +\tau )^4 (\alpha +\tau
   +1) (2 \alpha +2 \tau -1) (2 \alpha +2 \tau +1)^2 (2 \alpha +2 \tau +3)\right)\notag\\
 &&  +\frac{\left(\alpha ^2 \lambda +2 \alpha  \lambda  \tau -\alpha  \lambda -\alpha +\lambda
    \tau ^2-\lambda  \tau +1\right)^4}{(\alpha +\tau -1)^4 (\alpha +\tau )^4}.\notag
\end{eqnarray}

\noindent {\bf Moments of Case~2 analogous to Bessel}
\begin{eqnarray}
(u)_0 & = & 1, \quad (u)_1=\lambda,\quad (u)_2=\lambda ^2+\rho, \quad
(u)_3 =\lambda ^3+2 \lambda  \rho -\frac{\rho }{\alpha -2},\notag\\
(u)_4& =&\lambda ^4+3 \lambda ^2 \rho
 -\frac{2 \lambda  \rho }{\alpha -2}+\frac{\rho  \left((\alpha -2) (2 \alpha -5) \rho +2\right)}{(\alpha -2) (2
   \alpha -5)},\notag\\
(u)_5& =&\lambda ^5+4 \lambda ^3 \rho
-\frac{3 \lambda ^2 \rho }{\alpha-2}
+\frac{\lambda  \rho  \left(3 (\alpha -2) (2 \alpha -5) \rho +4\right)}{(\alpha
   -2) (2 \alpha -5)}-\frac{2 \rho  \left((\alpha -3) (2 \alpha -5) \rho +1\right)}{(\alpha -3) (\alpha -2) (2 \alpha -5)},\notag\\
(u)_6& =&\lambda ^6+5\lambda ^4 \rho -\frac{4 \lambda ^3 \rho }{\alpha -2}
+\frac{6 \lambda ^2 \rho  \left((\alpha -2) (2 \alpha -5) \rho +1\right)}{(\alpha -2) (2 \alpha -5)}
-\frac{2 \lambda  \rho  \left(3 (\alpha -3) (2 \alpha -5) \rho +2\right)}{(\alpha -3) (\alpha -2) (2 \alpha
   -5)}\notag\\
   &&+\frac{\rho  \left((\alpha -3) (\alpha -2)^2 (2 \alpha -7) (2 \alpha -5) \rho ^2+(\alpha -3) (2 \alpha -7)
   (6 \alpha -13) \rho +4 (\alpha -2)\right)}{(\alpha -3) (\alpha
   -2)^2 (2 \alpha -7) (2 \alpha -5)}.\notag
\end{eqnarray}

\noindent {\bf Moments of Case~3 analogous to Bessel}
\begin{eqnarray}
(u)_0 & = & 1, \quad (u)_1=\lambda,\quad (u)_2=\lambda ^2+\rho,\quad
(u)_3 = \lambda ^3+2 \lambda  \rho -\frac{\rho }{\alpha },\notag\\
(u)_4& =& \lambda ^4+3 \lambda ^2 \rho-\frac{2 \lambda  \rho }{\alpha }+\frac{\rho  \left(\alpha  (2 \alpha +1) \rho +2\right)}{\alpha  (2 \alpha
   +1)} ,\notag\\
(u)_5& =& \lambda ^5+4 \lambda ^3\rho-\frac{3 \lambda ^2 \rho }{\alpha }
+\frac{\lambda  \rho  \left(3 \alpha  (2 \alpha +1) \rho +4\right)}{\alpha  (2 \alpha
   +1)}-\frac{2 \rho  \left((\alpha +1) (2 \alpha +1) \rho +1\right)}{\alpha 
   (\alpha +1) (2 \alpha +1)},\notag\\
(u)_6& =&  \lambda ^6+5\lambda ^4 \rho  -\frac{4 \lambda ^3 \rho }{\alpha }+ \frac{6 \lambda ^2 \rho  \left(\alpha  (2 \alpha +1) \rho +1\right)}{\alpha  (2
   \alpha +1)}-\frac{2 \lambda  \rho  \left(3 (\alpha +1) (2 \alpha +1) \rho +2\right)}{\alpha  (\alpha +1) (2 \alpha +1)}\notag\\
   &&+\frac{\rho  \left(\alpha ^2 (\alpha +1) (2 \alpha +1) (2 \alpha +3) \rho ^2+(\alpha +1) (2 \alpha +3) (6
   \alpha +1) \rho +4 \alpha \right)}{\alpha ^2 (\alpha
   +1) (2 \alpha +1) (2 \alpha +3)}.\notag
\end{eqnarray}

\noindent {\bf Moments of Case~4 analogous to Bessel}
\begin{eqnarray}
(u)_0 & = & 1, \quad (u)_1=\frac{1}{(\alpha -1) \alpha }+\lambda,\quad 
(u)_2=\frac{\rho }{\alpha ^2 (2 \alpha -1) (2 \alpha +1)}+\frac{\left((\alpha -1) \alpha  \lambda +1\right)^2}{(\alpha -1)^2 \alpha
   ^2},\notag\\
(u)_3& =& \frac{\rho  \left(2 (\alpha -1) (\alpha +1) \alpha  \lambda +3 \alpha +1\right)}{(\alpha
   -1) \alpha ^3 (\alpha +1) (2 \alpha -1) (2 \alpha +1)}
   +\frac{\left((\alpha -1) \alpha  \lambda +1\right)^3}{(\alpha -1)^3 \alpha ^3},\notag\\
(u)_4& =& \frac{\rho ^2}{\alpha ^4 (2 \alpha -1)^2 (2 \alpha +1)^2}
+\frac{\left((\alpha -1) \alpha  \lambda +1\right)^4}{(\alpha -1)^4 \alpha ^4}\notag\\
   &&+\rho  \left(3 (\alpha -1)^2 (\alpha +1) (2 \alpha +1) (2 \alpha +3) \alpha ^2 \lambda
   ^2\right.\notag\\
   &&\left.+4 (\alpha -1) (2 \alpha +1)^2 (2 \alpha +3) \alpha  \lambda +25 \alpha ^3+37 \alpha ^2+22
   \alpha +6\right)\notag\\
   &&/\left((\alpha -1)^2 \alpha ^4 (\alpha +1) (2 \alpha
   -1) (2 \alpha +1)^2 (2 \alpha +3)\right),\notag\\
(u)_5& =& \frac{\rho ^2 \left(3 \alpha ^3 \lambda -3 \alpha  \lambda +5 \alpha +1\right)}{(\alpha
   -1) \alpha ^5 (\alpha +1) (2 \alpha -1)^2 (2 \alpha +1)^2}
   +\frac{\left((\alpha -1) \alpha  \lambda +1\right)^5}{(\alpha -1)^5 \alpha ^5}\notag\\
   &&+\Big(\rho  \left(45 \alpha ^4+4 (\alpha -1)^3 (\alpha +1) (\alpha +2) (2 \alpha +1) (2 \alpha +3) \alpha
   ^3 \lambda ^3\right.\notag\\
   &&+107 \alpha ^3+3 (\alpha -1)^2 (\alpha +2) (2 \alpha +1) (2 \alpha +3)
   (5 \alpha +3) \alpha ^2 \lambda ^2\notag\\
   &&\left.+122 \alpha ^2+2 (\alpha -1) (\alpha +2) \left(41
   \alpha ^3+77 \alpha ^2+50 \alpha +12\right) \alpha  \lambda +74 \alpha +12\right)\Big)\notag\\
   &&/\Big((\alpha -1)^3 \alpha ^5
   (\alpha +1) (\alpha +2) (2 \alpha -1) (2 \alpha +1)^2 (2 \alpha +3)\Big).\notag
\end{eqnarray}

\noindent {\bf Moments of Case~1 analogous to Jacobi}
\begin{eqnarray}
(u)_0 & = & 1, \quad (u)_1= \frac{(\alpha -\beta ) (\alpha +\beta )}{(\alpha +\beta +2 \tau ) (\alpha +\beta +2 \tau
   +2)}+\lambda,\notag\\
   (u)_2& =& \frac{4 \rho  (\tau +1) (\alpha +\tau +1) (\beta +\tau
   +1) (\alpha +\beta +\tau +1)}{(\alpha +\beta +2 \tau +1) (\alpha +\beta +2 \tau +2)^2
   (\alpha +\beta +2 \tau +3)}\notag\\
   &&+\frac{\left(\lambda  (\alpha +\beta +2 \tau ) (\alpha +\beta +2 \tau +2)+(\alpha -\beta ) (\alpha
   +\beta ) \right)^2}{(\alpha +\beta +2 \tau )^2
   (\alpha +\beta +2 \tau +2)^2}, \notag\\
(u)_3& =& \frac{\left(\lambda  (\alpha +\beta +2 \tau ) (\alpha +\beta +2 \tau +2)+(\alpha -\beta ) (\alpha
   +\beta ) \right)^3}{(\alpha +\beta +2 \tau )^3
   (\alpha +\beta +2 \tau +2)^3}\notag\\
   &&+\Big(4 \rho  (\tau +1) (\alpha +\tau +1) (\beta +\tau+1) (\alpha +\beta +\tau +1)\notag\\
   && \left(2 \lambda  (\alpha +\beta +2 \tau ) (\alpha +\beta +2 \tau +2) (\alpha +\beta +2 \tau
   +4)\right.\notag\\
   &&\left.+(\alpha -\beta ) (\alpha +\beta ) (3 \alpha +3 \beta +6 \tau +8) \right)\Big)\notag\\
 &&  /\Big((\alpha +\beta +2 \tau ) (\alpha +\beta +2 \tau +1)
   (\alpha +\beta +2 \tau +2)^3 (\alpha +\beta +2 \tau +3) \notag\\
   &&(\alpha +\beta +2 \tau +4)\Big).\notag\\
(u)_4& =&   \frac{\left(\lambda  (\alpha +\beta +2 \tau ) (\alpha +\beta +2 \tau +2)+(\alpha -\beta ) (\alpha
   +\beta ) \right)^4}{(\alpha +\beta +2 \tau )^4
   (\alpha +\beta +2 \tau +2)^4}\notag\\
   &&+\frac{16 \rho ^2 (\tau +1)^2 (\alpha +\tau +1)^2 (\beta
   +\tau +1)^2 (\alpha +\beta +\tau +1)^2}{(\alpha +\beta +2 \tau +1)^2 (\alpha +\beta
   +2 \tau +2)^4 (\alpha +\beta +2 \tau +3)^2}\notag\\
   && +\Big(4 \rho  (\tau +1) (\alpha +\tau +1) (\beta +\tau +1) (\alpha +\beta +\tau +1)\notag\\
   &&\Big(3 \lambda ^2 (\alpha +\beta +2 \tau )^2 (\alpha +\beta +2 \tau +3) (\alpha +\beta +2
   \tau +4) (\alpha +\beta +2 \tau +5) \notag\\
   &&(\alpha +\beta +2 \tau +2)^2\notag\\
   &&+8 (\alpha -\beta )
   (\alpha +\beta ) \lambda  (\alpha +\beta +2 \tau ) (\alpha +\beta +2 \tau +3)^2
   (\alpha +\beta +2 \tau +5)\notag\\
   && (\alpha +\beta +2 \tau +2)\notag\\
   &&+2 \Big(4 \tau ^5 \left(19 \alpha ^2+42 \alpha  \beta +96 \alpha +19 \beta ^2+96 \beta
   +100\right)\notag\\
   &&+2 \tau ^4 \left(25 \alpha ^3+95 \alpha ^2 \beta +224 \alpha ^2+95 \alpha 
   \beta ^2+480 \alpha  \beta +500 \alpha +25 \beta ^3+224 \beta ^2\right.\notag\\
   &&\left.+500 \beta+304\right)\notag\\
   &&+4 \tau ^3 \left(10 \alpha ^4+25 \alpha ^3 \beta +64 \alpha ^3+30 \alpha
   ^2 \beta ^2+224 \alpha ^2 \beta +240 \alpha ^2+25 \alpha  \beta ^3\right.\notag\\
   &&\left.+224 \alpha  \beta
   ^2+500 \alpha  \beta +304 \alpha +10 \beta ^4+64 \beta ^3+240 \beta ^2+304 \beta
   +112\right)\notag\\
   &&+2 \tau ^2 \left(19 \alpha ^5+30 \alpha ^4 \beta +92 \alpha ^4-5 \alpha ^3
   \beta ^2+192 \alpha ^3 \beta +220 \alpha ^3-5 \alpha ^2 \beta ^3\right.\notag\\
   &&\left.+200 \alpha ^2 \beta
   ^2+720 \alpha ^2 \beta +448 \alpha ^2+30 \alpha  \beta ^4+192 \alpha  \beta ^3+720
   \alpha  \beta ^2\right.\notag\\
   &&\left.+912 \alpha  \beta +336 \alpha +19 \beta ^5+92 \beta ^4+220 \beta
   ^3+448 \beta ^2+336 \beta +64\right)\notag\\
   &&+2 \tau  (\alpha +\beta ) \left(9 \alpha ^5+10
   \alpha ^4 \beta +60 \alpha ^4-15 \alpha ^3 \beta ^2+32 \alpha ^3 \beta +133 \alpha
   ^3\right.\notag\\
   &&\left.-15 \alpha ^2 \beta ^3-56 \alpha ^2 \beta ^2+87 \alpha ^2 \beta +144 \alpha ^2+10
   \alpha  \beta ^4+32 \alpha  \beta ^3+87 \alpha  \beta ^2+304 \alpha  \beta \right.\notag\\
   &&\left.+168
   \alpha +9 \beta ^5+60 \beta ^4+133 \beta ^3+144 \beta ^2+168 \beta +64\right)\notag\\
   &&+(\alpha +\beta )^2 \left(3 \alpha ^5+3 \alpha ^4 \beta +28 \alpha ^4-6 \alpha ^3 \beta ^2+4
   \alpha ^3 \beta +93 \alpha ^3-6 \alpha ^2 \beta ^3-48 \alpha ^2 \beta ^2\right.\notag\\
   &&\left.-53 \alpha ^2
   \beta +124 \alpha ^2+3 \alpha  \beta ^4+4 \alpha  \beta ^3-53 \alpha  \beta ^2-104
   \alpha  \beta +56 \alpha +3 \beta ^5+28 \beta ^4\right.\notag\\
   &&\left.+93 \beta ^3+124 \beta ^2+56 \beta
   +32\right)+8 \tau ^6 (7 \alpha +7 \beta +16)+16 \tau ^7 \Big)\Big)\notag\\
  &&/ \Big((\alpha +\beta +2 \tau )^2 (\alpha +\beta +2 \tau +1)
   (\alpha +\beta +2 \tau +2)^4 (\alpha +\beta +2 \tau +3)^2 \notag\\
   &&(\alpha +\beta +2 \tau +4)
   (\alpha +\beta +2 \tau +5)\Big).\notag
 \end{eqnarray}
 
\noindent {\bf Moments of Case~2 analogous to Jacobi}
\begin{eqnarray}
(u)_0 & = & 1, \quad (u)_1=\lambda,\quad (u)_2=\lambda ^2+\rho, \quad
(u)_3=\lambda ^3+2 \lambda  \rho-\frac{\rho  (\alpha -\beta )}{\alpha +\beta +2} ,\notag\\
(u)_4& =&\lambda ^4+3 \lambda ^2 \rho-\frac{2 \lambda  \rho  (\alpha
   -\beta )}{\alpha +\beta +2}\notag\\
   && +\frac{\rho  \left(\alpha ^2+\rho  (\alpha +\beta +2) (\alpha +\beta +3)-2 \alpha  \beta +\alpha +\beta
   ^2+\beta +2\right)}{(\alpha +\beta +2) (\alpha +\beta +3)},\notag\\
(u)_5& =&\lambda ^5+4 \lambda ^3 \rho -\frac{3 \lambda
   ^2 \rho  (\alpha -\beta )}{\alpha +\beta +2}\notag\\
   &&+\frac{\lambda  \rho  \left(2 \left(\alpha ^2-2 \alpha  \beta +\alpha +\beta ^2+\beta +2\right)+3 \rho  (\alpha
   +\beta +2) (\alpha +\beta +3)\right)}{(\alpha +\beta +2) (\alpha +\beta +3)}\notag\\
   &&-\frac{\rho  (\alpha
   -\beta ) \left(2 \rho  (\alpha +\beta +3) (\alpha +\beta +4)+\alpha ^2-2 \alpha  \beta +3 \alpha +\beta
   ^2+3 \beta +8\right)}{(\alpha +\beta +2) (\alpha +\beta +3) (\alpha +\beta +4)},\notag\\
(u)_6& =& \lambda ^6+5
   \lambda ^4 \rho -\frac{4 \lambda ^3 \rho  (\alpha -\beta )}{\alpha +\beta +2}\notag\\
   &&
    +\frac{3 \lambda ^2 \rho  \left(2 \rho  (\alpha +\beta +2) (\alpha +\beta +3)+\alpha ^2-2 \alpha  \beta +\alpha +\beta
   ^2+\beta +22\right)}{(\alpha +\beta +2) (\alpha +\beta +3)}\notag\\
   &&-\frac{2 \lambda 
   \rho  (\alpha -\beta ) \left(3 \rho  (\alpha +\beta +3) (\alpha +\beta +4)+\alpha ^2-2 \alpha  \beta +3 \alpha +\beta
   ^2+3 \beta +8\right)}{(\alpha +\beta +2) (\alpha +\beta +3) (\alpha +\beta
   +4)}\notag\\
   &&+\Big(\rho  \left(\rho  (\alpha +\beta +4) (\alpha +\beta +5) \right.\notag\\
   &&\left(3 \alpha ^3-3 \alpha ^2 \beta +9
   \alpha ^2-3 \alpha  \beta ^2-10 \alpha  \beta +8 \alpha +3 \beta ^3+9 \beta ^2+8
   \beta +8\right)\notag\\
   &&+(\alpha +\beta +2) \left(\alpha ^4-4 \alpha ^3 \beta +6 \alpha ^3+6
   \alpha ^2 \beta ^2-6 \alpha ^2 \beta +23 \alpha ^2-4 \alpha  \beta ^3-6 \alpha  \beta
   ^2\right.\notag\\
   &&\left.-34 \alpha  \beta +18 \alpha +\beta ^4+6 \beta ^3+23 \beta ^2+18 \beta
   +24\right)\notag\\
   &&\left.+\rho ^2 (\alpha +\beta +2)^2 (\alpha +\beta +3) (\alpha +\beta +4) (\alpha
   +\beta +5)\right)\Big)\notag\\
   &&/\Big((\alpha +\beta +2)^2 (\alpha +\beta +3) (\alpha +\beta +4) (\alpha +\beta
   +5)\Big).\notag
\end{eqnarray}

\begin{remark}
The moments of any regular form can be computed from the corresponding recurrence coefficients, $\beta_n$ and $\gamma_{n+1}$, using a general constructive method and algorithm introduced in \cite{Rocha-1991} for $d$-orthogonal polynomials, that can be applied here for $d=1$.
%
\end{remark}

\section{Conclusion}\label{Section6}

The recurrences established in Theorem~3 provide, for each of the ten canonical Laguerre--Hahn families of class zero, an explicit and algorithmic determination of the moment sequence from the first moment $(u)_0$ and the coefficients of the functional equation, namely $\Phi(x)$, $\psi(x)$, and $B(x)$. Their explicit structure allows a straightforward algorithm with implementation in {\it Mathematica$^{\circledR}$}, generating moments of arbitrarily high order in a systematic way, as illustrated in Section~\ref{Section5}.

Beyond class zero, the explicit determination of moments naturally arises for Laguerre-Hahn forms of higher class $s \geq 1$. Whether the approach developed here can be extended to those cases remains an open and challenging problem.


\section*{Addendum}

In this paper, we present the results of a scientific work originally initiated by  Pascal Maroni and Zélia da Rocha. This project combined theoretical results with their implementation in {\it Mathematica$^{\circledR}$}. As is often the case in scientific research, the work was temporarily suspended with the intention of resuming it at a later stage. 

Pascal Maroni was a highly prolific researcher, and his involvement in other projects and collaborations prevented him from returning to this subject. Unfortunately, he passed away in January 2024. His friends, collaborators, and admirers therefore decided to honor his memory by completing his ongoing works and publishing them in order to disseminate his ideas within the scientific community, while including Pascal Maroni as one of the authors.

In this spirit, Zélia da Rocha invited  Mohamed Khalfallah to complete the theoretical part, thereby allowing this work to be finally brought to completion.

\section*{Declarations}

\noindent  {\bf Data Availability} No datasets were generated or analysed during the current study.

\noindent  {\bf Conflicts of Interest} The authors have no conflict of interest to declare.

\noindent  {\bf Competing interests} The authors declare no competing interests.

\noindent {\bf Funding:} The third author was partially supported by CMUP, a member of LASI, which is financed by national funds through FCT -- Funda\c c\~ao para a Ci\^encia e a Tecnologia, I.P., under the projects with reference 
UID/00144/2025.


\bibliographystyle{plain}
\bibliography{bibliography-3}

\end{document}